\documentclass[a4paper,11pt]{amsart}

\usepackage[hidelinks]{hyperref}

\usepackage{amsmath,amsthm,amssymb,latexsym,amsfonts,wrapfig,mathtools}
\usepackage[latin1]{inputenc}
\usepackage[english,activeacute]{babel}
\usepackage{graphicx,xcolor}
\usepackage{enumerate}

\usepackage{parskip}
\usepackage{hyperref}
\usepackage{cleveref}

\usepackage{todonotes}

\newcommand{\norm}[1]{\left\lVert#1\right\rVert}
\allowdisplaybreaks

\def \N{\mathbb N}
\def \Z{\mathbb Z}

\def \R{\mathbb R}
\def \C{\mathbb C}

\def \P{{\mathbb P}}
\def \pa{{\partial}}
\def \Fc{{\mathcal F}}

\def \O{\mathcal{O}}
\def \T{\mathbb{T}}

\def \E{\mathbb{E}}

\numberwithin{equation}{section}

\theoremstyle{plain}
\newtheorem{thm}{Theorem}[section]

\newtheorem{lem}[thm]{Lemma}
\newtheorem{prop}[thm]{Proposition}
\newtheorem{cor}[thm]{Corollary}
\theoremstyle{definition}

\theoremstyle{remark}
\newtheorem{rk}[thm]{Remark}

\theoremstyle{plain}

\theoremstyle{remark}

\theoremstyle{plain}

\newcommand{\ee}[1]{{\mathbb E}\left[ #1 \right]}
\newcommand{\pp}[1]{{\mathbb P}\left(#1\right)}

\title{Large deviations principle for the cubic NLS equation}

\author[M. A. Garrido]{Miguel Angel Garrido}
\address{Department of Statistics, Columbia University, New York, NY}
\email{mag2319@columbia.edu}

\author[R. Grande]{Ricardo Grande}
\address{D\'epartement de math\'ematiques et applications, \'Ecole normale sup\'erieure, CNRS, PSL University, 75005 Paris, France}
\email{ricardo.grande@ens.fr} 
\thanks{R. Grande  was  supported by the  Simons Collaboration Grant on Wave Turbulence}

\author[K. M. Kurianski]{Kristin M. Kurianski}
\address{Department of Mathematics, California State University Fullerton, Fullerton, CA}
\email{kkurianski@fullerton.edu}

\author[G. Staffilani]{Gigliola Staffilani}
\address{Department of Mathematics, Massachusetts Institute of Technology, Cambridge, MA}
\email{gigliola@math.mit.edu}
\thanks{G. Staffilani was  supported by the NSF grant 
DMS-1764403, DMS-2052651 and the Simons Collaboration Grant on Wave Turbulence}

\begin{document}
	
	\begin{abstract}
In this paper, we present a probabilistic study of rare phenomena of the cubic nonlinear Schr\"odinger equation on the torus in a weakly nonlinear setting. This equation has been used as a model to numerically study the formation of rogue waves in deep sea. Our results are twofold: first, we introduce a notion of criticality and prove a Large Deviations Principle (LDP) for the subcritical and critical cases. Second, we study the most likely initial conditions that lead to the formation of a rogue wave, from a theoretical and numerical point of view. Finally, we propose several open questions for future research.
\end{abstract}

	\maketitle

\section{Introduction}

\subsection{Motivation and background}

In the last few decades there has been considerable research in the field of dispersive equations, and particularly on the nonlinear Schr\"odinger (NLS) equation:
\begin{equation}\label{eq:intro_NLS}
\begin{cases}
 i\pa_t u + \Delta u = \mu \, |u|^2\, u, \qquad \mu=\pm 1,\\
 u|_{t=0}=u_0\, .
 \end{cases}
\end{equation}
One of the most fundamental questions is that of existence and uniqueness of solutions. The question of well-posedness on $\R^d$ was extensively studied in the 1980s and the 1990s for subcritical initial data  \cite{CazenaveWeissler,GinibreVelo,Kato,Tsutsumi} and later for energy and mass critical initial data \cite{B-crit, Gri, I-team, Vis, Dod}.
Well-posedness on the torus $\T^d$ received much attention in the 1990s, mainly due to Bourgain's breakthrough paper \cite{Bourgain-lwp}, which was followed by  work in  \cite{Herr-all, Io-Pa, Yue}.
These results are \emph{deterministic}, in the sense that one shows that a solution exists for \emph{every} initial datum $u_0$ in a certain space, typically Sobolev spaces such as $H^s$ where the regularity $s$ must be larger or equal to a critical value $s_c$.

Beyond existence of solutions, it is important to understand the long-time behaviour of such solutions. In $\R^d$ solutions to \eqref{eq:intro_NLS} with $\mu=1$ and in the critical and subcritical case scatter (sometimes under some technical conditions), see for example \cite{GinibreVelo, I-team, Vis, Dod, TVZ, YX}, while in the super critical case blow-up may occur, see the recent work \cite{MRRS}. In the case $\mu=-1$, some solutions to \eqref{eq:intro_NLS} blow up in finite time, and therefore they cannot be extended indefinitely, see for example  \cite{MR-blow, MRRS}. In contrast to this generic behavior, certain types of 
initial data give rise to traveling waves and soliton solutions, see \cite{Alejo} and the references therein. These solutions display an unusual long-term behavior, and as a result they are considered \emph{rare phenomena}.

On the torus $\T^d$, scattering is not possible so one expects different asymptotics. In particular, it is of great physical interest  to analyse  phenomena related to   weak turbulence, and more  precisely that  of  energy transfer. With this we mean the expected behaviour of certain initial data with frequency localised support that  evolve into solutions whose support lives mostly in higher frequencies. This phenomenon is connected to the question of asymptotic growth of Sobolev norms of solutions to the Schr\"odinger equation whose study was started by Bourgain \cite{Bourgain-growth} and later continued in \cite{Sta-p, Soh, StW, DG, Deng, HPSW, PTV17, BertiMaspero, CarlesFaou, Iteam, GHHMP}  among others. In addition to such theoretical results, there are many interesting simulations and experimental results about  transfer of energy to high frequencies, see for example  \cite{MMT, CSS, HP, HPSW}.

In 1988, the work of Lebowitz, Rose, and Speer \cite{LRS}, followed by Bourgain's papers on invariant measures\footnote{There have been many recent advances on this topic, see \cite{Deng-Nahmod-Yue,Deng-Nahmod-Yue-2} and the references therein.}  \cite{Bourgain-inv1,Bourgain-inv2}, started the probabilistic study of solutions to the nonlinear Schr\"odinger equation in a context in which until then harmonic and Fourier analysis had been the predominant and preferred tools of investigation. In such works, one typically considers random initial data, whose Fourier coefficients are normally distributed. This viewpoint has been extensively used to obtain \emph{generic} results which hold for  almost every initial datum (i.e. almost surely). For example, this additional flexibility led in many cases to an improved well-posedness theory, see for example \cite{DLM, Deng-Nahmod-Yue,Deng-Nahmod-Yue-2}. The random data theory has also been fundamental in the study of the long-term behaviour of solutions to \eqref{eq:intro_NLS} on $\T^d$ that we mentioned above. Indeed, one of the most interesting lines of research in this direction is the derivation of the Wave Kinetic Equation (WKE), which connects the study of the {\it typical} size of the Fourier coefficients of the solution to \eqref{eq:intro_NLS} to the study of solutions to a kinetic equation \cite{BGHS, CollotGermain, CollotGermain2, DengHani,DengHani2, StaffilaniTran}. All such works make the most of probabilistic estimates to argue that non-generic behavior is unlikely, but few of them focus on the study of such non-generic behavior.

In this paper, we study rogue waves as an interesting and still not well understood example of rare phenomena. Oceanographers generally define rogue waves to be deep-water waves whose amplitude exceeds twice the characteristic wave height expected for the given surface conditions \cite{VandenEijnden,Dysthe-review}. Over the past decades, there has been considerable interest in modelling rogue waves, both because analytically their origins are still mysterious and  because their  sudden appearances   pose a threat for a variety of naval infrastructure.

Several evolution equations have been classically used as models for the study of rogue waves. Two of the most popular models are given by the cubic Schr\"odinger equation and the Dysthe equation. The Dysthe equation has been the subject of many experimental and computational works, for instance \cite{VandenEijnden,Dysthe,FarazmandSapsis}; as well as some theoretical results, e.g. \cite{GKS,Saut}. A few variants of the NLS equation have also been used to model rogue wave formation, see \cite{OnoratoVE,Onorato-mechanism} among others. Note that all these models are derived from the incompressible Navier-Stokes equation, after  truncating an asymptotic expansion of the modulation of a wavetrain at various orders \cite{Dysthe,Saut}. 

There are several theories to explain the precise mechanism for the formation of rogue waves \cite{Onorato-mechanism}, but this phenomenon is still the subject of much debate. In the case of the focusing NLS equation in $\R$, Bertola and Tovbis obtained a precise description of a deterministic class of solutions near their peak \cite{BertolaTovbis}. One of the goals of this paper is to provide a characterization of rogue waves of the weakly nonlinear NLS equation on $\T=[0,2\pi]$ with random initial data.

The results in this paper are inspired by some recent work of Dematteis, Grafke, Onorato, and Vanden-Eijnden \cite{OnoratoVE,VandenEijnden,VandenEijnden2}. In these papers, the authors conjecture the existence of a Large Deviation Principle that may be used to predict and study the formation of rogue waves. More precisely, they consider an equation akin\footnote{Technically, the equations in \cite{OnoratoVE} and \cite{VandenEijnden} are different, but the methods and the underlying theory that the authors propose are similar, see \cite{VandenEijnden2}. Interestingly, their numerics and experimental results seem to hold regardless of the equation, with some natural modifications.} to \eqref{eq:intro_NLS} with $\mu=-1$ on a circle of fixed size $L$, with random initial data $u_0$
of the form
\begin{equation}\label{eq:intro_dataVE}
u_0(x)=\sum_{|k|\leq N} \vartheta_k e^{ik\cdot x}.
\end{equation}
Here, the set of initial data is parametrized by $\vartheta=(\vartheta_k)_{|k|\leq N}\in \C^{2N+1}$. In order to make $u_0$ random, the authors endow the space of parameters $\C^{2N+1}$ with a probability measure such that $\left\{\vartheta_k\right\}_{|k|\le N}$ are independent and identically distributed (i.i.d.) complex Gaussian random variables with $\E \vartheta_k= \E\vartheta_k^2=0$ and $\E|\vartheta_k|^2=c_k^2$, for some fast-decaying $c_k>0$. 

For $t>0$ and $z>0$ fixed, the set of initial data that generate a rogue wave of height at least $z$ at time $t$ is given by
\begin{equation}\label{Dtz}
\mathcal{D}(t,z):=\left \lbrace (\vartheta_k)_{|k|\leq N}\in \C^{2N+1} \Big| \sup_{x\in [0,L]} |u(t,x \, | \, \vartheta)| >z \ \right \rbrace 
\end{equation}
where $u$ is the solution to the equation with initial data \eqref{eq:intro_dataVE}.

In order to quantify the likelihood of $\mathcal{D}(t,z)$, Dematteis et al.\ propose a theoretical framework that allows them to conclude a large deviations principle (LDP), provided that the minimization problem
\begin{equation}\label{eq:intro_minimizerVE}
\vartheta^{\ast}(z)= \underset{\vartheta\in\mathcal{D}(t,z)}{\mathrm{argmin}} \,\, \mathcal{I}(\vartheta).
\end{equation}
has a unique solution, where
\[
\mathcal{I}(\vartheta)  = \max_{y\in\C^{2N+1}} [\,\langle y,\vartheta\rangle - S(y) ] \qquad \mbox{and}\quad 
S(y) = \log \E e^{\langle y, \vartheta\rangle}.
\]
In particular, for a fixed $t>0$ and as $z\rightarrow \infty$, they claim that
\begin{equation}\label{eq:intro_LDP_VE}
\log \P ( \mathcal{D}(t,z)) = -\mathcal{I}(\vartheta^{\ast}(z))+ o(1).
\end{equation}

The goal of this paper is to prove rigorously conjecture \eqref{eq:intro_LDP_VE} in the case of the NLS equation with a weak nonlinearity. The application of 
the theoretical framework proposed by Dematteis et al.\ requires verifying a set of strong assumptions that are very difficult to check in practice, such as bounds on the gradient $\nabla_{\vartheta}\sup_{x\in [0,L]} |u(t,x \, | \, \vartheta)|$ or the convexity of the set $\mathcal{D}(t,z)$, so that the existence of a unique minimizer for \eqref{eq:intro_minimizerVE} is guaranteed.


\noindent For these reasons, we follow a different approach that allows us to establish an LDP of the form
\begin{equation}\label{eq:intro_ourLDP}
\log \P \left( \sup_{x\in\T} |u(t,x)|>z\right) = - I(z) +o(1), \quad \mbox{as}\ z\rightarrow\infty ,
\end{equation}
where we can derive the rate function $I$ without having to solve the minimization problem in \eqref{eq:intro_minimizerVE}. In \eqref{eq:intro_ourLDP}, $u$ is the solution to the equation with initial data of the form \eqref{eq:intro_dataVE} with an infinite number of Fourier modes ($N=\infty$). See \Cref{thm:mainLDP} below for a precise statement.

In the second part of the paper, we investigate the following question: ``Conditioned on the fact that we observe a rogue wave of height at least $z >0$ at time $t >0$, what is the most likely initial datum that generates such phenomenon?'' To answer this question, we recover the minimization problem given by the right-hand side of \eqref{eq:intro_minimizerVE}, and show that it admits a two-parameter family of solutions. In addition, we can construct neighborhoods $\mathcal{U}(t,z)$ around the family of minimizers whose probability is almost identical to that of $\mathcal{D}(t,z)$. See \Cref{thm:MLE} below for a precise statement.

 
\subsection{Statement of results}

In this paper, we focus on the cubic Schr\"odinger equation on the torus $\T=[0,2\pi]$ with a weak nonlinearity:
\begin{equation}\label{eq:intro_weakNLS}
\begin{cases}
i\pa_t u + \Delta u = \mu \, \varepsilon^{\alpha} |u|^2\, u\\
u(t,x)|_{t=0}=u_0
\end{cases}
\end{equation}
for  $\alpha\geq 0$ and $\mu=\pm 1$, and where the parameter $\varepsilon>0$ controls the size of the nonlinearity. We consider 
random initial data of the form:
\begin{equation}\label{eq:intro_data}
    u_0 (x)= \sum_{k\in\Z} c_k \eta_k e^{ikx}
\end{equation}
where
\begin{enumerate}[(1)]
    \item the coefficients have the form $c_k=a\,e^{-b|k|}$, or $c_k=a\,e^{-b|k|^2}$, for all $k\in\Z$, where $a,b>0$ are fixed; and
    \item $\{\eta_k\}_{k\in\Z}$ are independent, identically distributed, complex Gaussian random variables with $\E\eta_k=0$, $\E\eta_k\eta_j=0$ and $\E \eta_k \overline{\eta}_j =\delta_{kj}$ for $k,j\in\Z$.
\end{enumerate}

Assumption (1) is based on the choice of coefficients in \cite{VandenEijnden}, but our results extend to other families of coefficients, see \Cref{rk:coeff}. In the sequel we assume that $a=1$ since we can rescale the equation if that is not the case.

Our first result is a rigorous proof of \eqref{eq:intro_ourLDP} for solutions to \eqref{eq:intro_weakNLS} with $\alpha=2$ and initial data of the form \eqref{eq:intro_data}.

\begin{thm}[Large Deviation Principle]\label{thm:mainLDP} Consider the NLS equation on the circle $\T=[0,2\pi]$:
\begin{equation}\label{eq:NLS_mainLDP}
\begin{cases}
i\pa_t u + \Delta u = \varepsilon^2 |u|^2\, u \\
u(0,x)=\sum_{k\in\Z} c_k \eta_k e^{ik x} 
\end{cases}
\end{equation}
with initial data as in \eqref{eq:intro_data}. Consider the probability of seeing a large wave of height $z(\varepsilon):=z_0\, \varepsilon^{-1/2}>0$ at time $t>0$ and fixed $z_0>0$. If $t\lesssim \varepsilon^{-1}$ we have that
\begin{equation}\label{eq:mainLDP}
\lim_{\varepsilon\rightarrow 0^{+}} \varepsilon\,\log \P \left( \sup_{x\in\T} |u(t,x)| > z_0 \, \varepsilon^{-1/2} \right) =- \frac{z_0^2}{\sum_{k\in\Z} c_k^2}.
\end{equation}
\end{thm}
\begin{rk}
The requirement that $t\lesssim \varepsilon^{-1}$ is connected to the criticality of the problem. This notion of criticality is discussed in \Cref{rk:criticality}. The proof of \Cref{thm:mainLDP} is based on approximating the solution $u$ by a function that satisfies \eqref{eq:mainLDP}. For subcritical times $t\ll \varepsilon^{-1}$, a linear approximation suffices. For critical times $t=\O (\varepsilon^{-1})$, the linear approximation is not enough, and we must use a resonant approximation instead.
\end{rk}

\begin{rk}\label{rk:coeff} The coefficients $c_k$ given in \eqref{eq:intro_data} are based on those in \cite{VandenEijnden}, which are chosen to model observations in the North Sea. Finding the optimal family of coefficients (in terms of decay) so that \eqref{eq:mainLDP} holds is an interesting open question.
\end{rk}
\begin{rk}
We have stated \Cref{thm:mainLDP} for the defocusing equation, but our results extend to the focusing equation. In fact the sign of the nonlinearity has no effect in our results since we are in a weakly nonlinear regime. As we will show below, our results are strongly based on the special form that a linear solution takes in the 1D setting, as well as the resonant set in this dimension. We speculate though that a possible extension to a 2D setting may be possible for irrational tori. There in fact, it has been proved \cite{HPSW} that the irrationality of the torus completely decouples  the resonant set into two 1D resonant sets, one for each  coordinate. 
\end{rk}

\begin{rk} The cubic NLS equation on $\T$ is an integrable system, however our results are related to neither integrability nor blow-up. We note that our  initial data live in $\cap_{s>0} H^s (\T)$ almost surely\footnote{By Bourgain's work \cite{Bourgain-lwp} one first obtains local well-posedness in $L^2$ in an interval of time depending on the $L^2$ norm of the initial data, then by using conservation of mass, the local 
well-posedness and preservation of regularity  guarantee the result for all times in $\cap_{s>0} H^s (\T)$. More on this in Remark \ref{rk:GWP} below.}, thus a unique global solution exists almost surely. The question of rogue waves is related to the $L^{\infty}(\T)$-norm of the solutions and their growth over time. Most solutions to the equation will barely grow over time. As shown in \Cref{sec:MLE}, rogue waves will typically peak at a certain point in space and time and decrease in size quickly after.
\end{rk}

The factor of $\varepsilon^2$ in front of the cubic nonlinearity in  \Cref{thm:mainLDP} is motivated by the standard notation in the PDE literature. However, our techniques can handle the more general scenarios presented below. For the sake of concreteness, we will limit ourselves to the setting in \Cref{thm:mainLDP} in the rest of the paper.

\begin{thm}\label{thm:generalLDP} Consider the NLS equation on the torus $\T=[0,2\pi]$:
\[
\begin{cases}
i\pa_t u + \Delta u = \varepsilon^{\alpha}\, |u|^2\, u\\
u(0,x)=\sum_{k\in\Z} c_k \eta_k e^{ik x} ,
\end{cases}
\] 
where $\alpha>0$, and initial data as in \eqref{eq:intro_data}. Consider the probability of seeing a large wave of height $z(\varepsilon):=z_0 \varepsilon^{-1/2}$ at time $t\sim \varepsilon^{-\beta}$ for some $\beta\in\R$. Then we have that:
\begin{itemize}
\item If $\alpha-1>\beta$, then \eqref{eq:mainLDP} holds.
\item If $\alpha-1 =\beta$ and $\beta>0$, then \eqref{eq:mainLDP} holds.
\end{itemize}
\end{thm}
\begin{rk}\label{rk:criticality}
Let us motivate the range of parameters in \Cref{thm:generalLDP}. First we rewrite \eqref{eq:intro_weakNLS} by using Duhamel formula:
\[ 
u(t) = e^{it\Delta} u_0 + u_{\mathrm{nl}}(t)=e^{it\Delta} u_0+ \varepsilon^{\alpha}\, \int_0^t e^{i(t-t')\Delta} |u|^2 u (t')\, dt'.
\]
At a certain fixed time $t$ (independent of $\varepsilon$) we impose the condition of seeing a rogue wave of size $\varepsilon^{-1/2}$, i.e. $\norm{u(t)}_{L^{\infty}(\T)}>\varepsilon^{-1/2}$. Next we estimate the nonlinear term $u_{\mathrm{nl}}(t)$ in some space $X$ which embeds continuously into $L^{\infty}(\T)$. In this paper, we use various Fourier-Lebesgue spaces such as $X=\Fc L^{0,1}$, but other choices such as $X^{s,b}$ spaces introduced by Bourgain in \cite{Bourgain-lwp} are possible. Then we have that 
\[ \varepsilon^{-1/2} < \norm{u(t)}_{L^{\infty}_x} \leq \norm{u(t)}_{X},\]
and
\[  \norm{u_{\mathrm{nl}}(t)}_{X}= \varepsilon^{\alpha}\,\norm{\int_0^t e^{i(t-t')\Delta} |u|^2 u (t')\, dt'}_{X} \lesssim  \varepsilon^{\alpha}\,\norm{u}_{X}^3.\]
Note that the right-hand side is at least as large as $\varepsilon^{\alpha-3/2}$ (but it might be much larger!). Assuming that this trilinear estimate is sharp, we say that the problem is subcritical if $\alpha -3/2 < -1/2$, critical if $\alpha -3/2 = -1/2$ and supercritical otherwise. This explains the range of parameters in \Cref{thm:generalLDP} when $\beta=0$. The notion of criticality may be extended to times $t$ which depend on $\varepsilon$ by rescaling the equation.
\end{rk}
\begin{rk} 
Note that \Cref{thm:mainLDP} corresponds to the choice $\alpha=2$. An interesting open question would be to understand the critical case $(\alpha,\beta)=(1,0)$, as well as any supercritical cases such as $\alpha=0$ (fully non-linear regime). In fact, in the case $\alpha = 0$, the numerical experiments in \cite{OnoratoVE,VandenEijnden,VandenEijnden2} suggest that the probability of seeing a rogue wave decays at a slower speed. This would correspond to a non-Gaussian scaling in \eqref{eq:mainLDP} with a sub-quadratic rate function $I(\cdot)$ in \eqref{eq:intro_ourLDP}.

Our work constitutes a first step towards solving this more challenging problem, and proposes a framework to tackle the non-linear problem via approximating the solution of \eqref{eq:NLS_mainLDP} and deriving explicit probability bounds for this approximation.
\end{rk}

\Cref{thm:mainLDP} allows us to quantify the probability that we observe a rogue wave. Despite being rare, rogue waves can actually happen. Our next result investigates the following question: ``Conditioned on the fact that we observe a rogue wave of height at least $z>0$ at time $t>0$, what is the most likely initial datum that generates such phenomenon?'' To answer this question, we introduce the set 
\[
\mathcal{H} =\{(r_k,\phi_k)_{k\in\Z} \mid r_k\geq 0,\ \phi_k\in [0,2\pi)\}, 
\]
and endow it with a probability measure such that $r_k e^{i\phi_k}$ are distributed as independent complex Gaussian as in \eqref{eq:intro_data}. Then, we can parametrize the space of initial data as follows:
\[
\vartheta = (r_k,\phi_k)_{k\in\Z} \in \mathcal{H} \mapsto u_0(x\mid \vartheta )=\sum_{k\in\Z} c_k r_k e^{ikx+i\phi_k}.
\]
As we explain in \Cref{sec:MLE}, the question of finding the most likely initial data is connected to a certain minimization problem over the set
\begin{equation}\label{eq:intro_set_roguewaves}
\mathcal{D}(t,z_0\varepsilon^{-1/2}) = \left\lbrace\vartheta\in\mathcal{H}\ \Big |\ \sup_{x\in\T} |u(t,x\mid \vartheta)|> z_0\varepsilon^{-1/2}\right\rbrace.
\end{equation}
The solution to this minimization problem when $u$ is the solution of the linear Schr\"odinger equation yields a two-parameter family of initial data given by:
\begin{equation}\label{eq:intro_minimizers}
\begin{split}
r_k^* & := \frac{c_k z_0\varepsilon^{-1/2}}{\sum_{j\in\Z} c_j^2},\\
\phi^*_k & := \phi^*_0 - kx^* + k^2 t \, \left(\text{mod } 2\pi\right),
\end{split}
\end{equation}
where $x^{\ast}\in\T$ parametrizes the position of the peak of the rogue wave, and $\phi^{\ast}_0\in[0,2\pi)$ parametrizes the phase of the Fourier mode associated with $c_0$.

The main result in \Cref{sec:MLE} shows that this special family of initial data concentrates as much probability as the entire set of rogue waves \eqref{eq:intro_set_roguewaves}. In order to state this theorem, we construct a small neighborhood $\mathcal{U}(\varepsilon)$ around this family, given by the $\vartheta =(r_k,\phi_k)_{k\in\Z}$ satisfying:
\begin{equation}\label{eq:intro_nhood}
\begin{split}
r_k^{\ast} \leq r_k & \leq r_k^{\ast} + \varepsilon^{2/5} \qquad \mbox{for all}\ |k|\leq -\frac{1}{b} \log \left( \frac{2\sum_{j\in\Z} c_j^2}{z_0}\, \varepsilon^{9/10}\right),\\
 |\phi_k -\phi_k^{\ast}| & \leq \varepsilon\qquad \mbox{for all}\ |k|\leq \varepsilon^{-1/2}\ \mbox{except $k=0,1$}.
 \end{split}
\end{equation}

\begin{thm}\label{thm:MLE}
Consider the set $\mathcal{U}(\varepsilon)$ given by \eqref{eq:intro_minimizers}-\eqref{eq:intro_nhood}. Then $\mathcal{U}(\varepsilon)$ satisfies the same LDP, \eqref{eq:mainLDP}, as the set of rogue waves \eqref{eq:intro_set_roguewaves}. Moreover, $\mathcal{U}(\varepsilon)$ is almost entirely contained in the set  $\mathcal{D} (t,z_0\varepsilon^{-1/2}-\varepsilon)$. More precisely,
\begin{equation}
\log \P \left( \mathcal{U}(\varepsilon) -\mathcal{D} (t,z_0\varepsilon^{-1/2} - \varepsilon)\right) \lesssim - \exp (c \varepsilon^{-1/2})) \qquad \mbox{as}\ \varepsilon\rightarrow 0^{+}.
\end{equation}
\end{thm}
\begin{rk}
The construction of the set $\mathcal{U}(\varepsilon)$ in \Cref{thm:MLE} is based on the linear Schr\"odinger equation. However, one can adapt this construction to tackle the nonlinear equation up to times given by \Cref{thm:mainLDP}. See \Cref{sec:MLE} for more details.
\end{rk}

Finally, we list a series of interesting open questions at the end of \Cref{sec:MLE}. Furthermore, we complement this theoretical analysis with some simulations to illustrate how the family of initial data given by \eqref{eq:intro_minimizers} behaves as time passes.

\subsection{Ideas of the proof}

\subsubsection{Main ideas in \Cref{thm:mainLDP}}
At a conceptual level, the proof of \Cref{thm:generalLDP} and that of \Cref{thm:mainLDP} are similar, so let us focus on the latter. The main ideas of the proof of \Cref{thm:mainLDP} are: (a) finding a good approximation $u_{\mathrm{app}}$ to the solution to the NLS equation \eqref{eq:NLS_mainLDP}, and (b) proving a LDP akin to \eqref{eq:mainLDP} for $u_{\mathrm{app}}$. For small times $t\ll \varepsilon^{-1}$, the linear flow constitutes a good approximation, whereas larger times require the use of a \emph{resonant approximation}\footnote{Related to the remark above about the 2D problem, for  irrational 2D tori it  is exactly in the resonant part of the NLS equation that we see the decoupling into two 1D resonant sets and hence we may be able to have good approximations also in 2D using the 1D tools developed below.}.

The resonant approximation is an explicit solution to an equation similar to \eqref{eq:NLS_mainLDP}, where we replace the nonlinearity $\varepsilon^2 |u|^2 u$ by a different cubic nonlinearity that captures the main contribution of $\varepsilon^2 |u|^2 u$ to the equation over timescales $t=\O (\varepsilon^{-1})$. This approximation has the form:
\begin{equation}\label{eq:intro_resonant}
u_{\mathrm{app}} (t,x) =e^{2it \varepsilon^2 \, M}\,  \sum_{k\in\Z} c_k \eta_k \, e^{ikx +i \varepsilon^2 t\, c_k^2 |\eta_k|^2- it |k|^2}
\end{equation}
where $M=\norm{u(t)}_{L^2(\T)}^2$ is the mass, which is conserved. Further details are provided in \Cref{subsec:resonant}. Such approximations have been used by Carles, Dumas and Sparber \cite{CarlesDumasSparber} and Carles and Faou \cite{CarlesFaou} in the context of
instability of the periodic NLS equation and energy cascades, respectively.

Intuitively, the reason why $\sup_{x\in\T}|u_{\mathrm{app}}(t,x)|$ remains sub-Gaussian is that nonlinear effects are restricted to the phases of the Fourier coefficients, while the moduli are unchanged. It is unclear whether this holds in other regimes beyond those given by \Cref{thm:generalLDP}, especially so in the case of the focusing NLS equation.

Both the linear and resonant approximations satisfy an LDP of the form \eqref{eq:mainLDP}. The first ingredient in the proof of \Cref{thm:mainLDP} is a general result in Large Deviations Theory known as the G\"artner-Ellis theorem, see \Cref{sec:linearLDP} as well as \Cref{thm:resonantLDP}. This classic result in Probability Theory uses convex analysis to state that an LDP for a family of probability measures, $\left\{\mu_{\varepsilon}\right\}_{\varepsilon > 0}$, holds provided that the limit of their re-scaled cumulant-generating functions exists and enjoys some good regularity properties.

The second ingredient in the proof of \Cref{thm:mainLDP} is a bootstrapping  argument, which shows that a large difference between the actual solution to \eqref{eq:NLS_mainLDP} and our approximation must be a result of an even larger initial datum $u_0$. Such large initial data are very unlikely, which allows us to replace the actual solution $u(t,x)$ in \eqref{eq:mainLDP} by the approximation $u_{\mathrm{app}}(t,x)$ up to a negligible error. As a result one can extend the LDP from $u_{\mathrm{app}}$ to $u$.

\subsubsection{Main ideas in \Cref{thm:MLE}}

The goal of this result is to show that the most likely way in which rogue wave arise in this weakly nonlinear context is due to phase synchronization\footnote{Also known as ``constructive interference'' in the Physics literature.}. The construction of the one-sided neighborhood $\mathcal{U}(\varepsilon)$ in \eqref{eq:intro_nhood} is based on two-scales: we control the moduli of $\O (|\log \varepsilon|)$ Fourier modes, but we force $\O (\varepsilon^{-1/2})$ phases to almost synchronize at a certain point in space and time.

In \Cref{thm:LDP_U}, we exploit the fact that our construction is explicit to show that $\mathcal{U}(\varepsilon)$ satisfies the same LDP as that in \Cref{thm:mainLDP} by an elementary argument. The key result, however, is to show that the elements in $\mathcal{U}(\varepsilon)$ are almost exclusively rogue waves. To prove this we must quantify the probability that the Fourier modes that we do not control in $\mathcal{U}(\varepsilon)$ work against us and create cancellation. Loosely speaking, this is achieved by proving the estimate
\[ 
\sup_{x\in\T} |u(t,x)|^2 \geq \left(\sum_{|k|\lesssim |\log \varepsilon|} c_k |\eta_k|\right)^2 + E(\varepsilon)
\]
for some explicit error function $E(\varepsilon)$. The key argument in the proof of \Cref{thm:MLE} is a careful analysis of this error function, which shows that it is very close to zero with very high probability.

\subsection{Outline} 
In \Cref{sec:linearLDP}, we develop a LDP for the linear Schr\"odinger equation \eqref{eq:intro_weakNLS} (with $\mu=0$). In \Cref{sec:subcritical}, we prove \Cref{thm:mainLDP} for subcritical times. In \Cref{sec:critical}, we introduce the resonant approximation and prove \Cref{thm:mainLDP} for critical times. Finally, in \Cref{sec:MLE} we motivate and study the minimization problem that leads to \eqref{eq:intro_minimizers}. Moreover, we prove \Cref{thm:MLE} and conduct numerical simulations about the ``most typical'' rogue waves.

\subsection{Notation}
We write $A\lesssim B$ to indicate an estimate of the form $A\lesssim C B$ for some positive constant $C$ which may change from line to line. The inequality $A\lesssim_d B$ indicates that the implicit constant $C$ depends on $d$.

We also employ the big $\O$ notation $A=\O_d(B)$, which means that $A\lesssim B$ as $d\rightarrow 0$. For a real number $a$, the notation $a-$ means $a-\varepsilon$ for $0<\varepsilon\ll 1$ small enough. Similarly, $a+$ means $a+\varepsilon$ for $0<\varepsilon\ll 1$ small enough.

For $1\leq p\leq \infty$ we often work in the space $L^p ([a,b])$, which consists of functions $f:[a,b]\rightarrow \C$ such that 
\[
\norm{f}_{L^p} = \left( \int_{a}^{b} |f(x)|^p\, dx \right)^{1/p} <\infty ,
\]
with the usual modifications when $p=\infty$. Similarly, $\ell^p$ consists of sequences $f=(f_k)_{k\in\Z}$, $f_k\in\C$, such that
\[
\norm{f}_{\ell^p} = \left( \sum_{k\in\Z} |f_k|^p \right)^{1/p} <\infty .
\]

In \Cref{sec:subcritical} and \Cref{sec:critical} we work in Fourier-Lebesgue spaces $\Fc L^{s,p}=\Fc L^{s,p} (\T)$ with $s\in\R$ and $1\leq p\leq \infty$. These are spaces of functions $f:\T \rightarrow\C$ whose Fourier coefficients $(f_k)_{k\in\Z}$ have the following finite norm:
\[
\norm{f}_{\Fc L^{s,p}} = \norm{ \langle k\rangle^{s} f_k}_{\ell_k^p}= \left(\sum_{k\in\Z} \langle k\rangle^{ps}\, |f_k|^p \right)^{1/p} = \left(\sum_{k\in\Z} (1+|k|)^{ps}\, |f_k|^p \right)^{1/p}.
\]
Finally, we denote $H^s(\T)=\mathcal{F}L^{s,2}(\T)$ as usual.

\subsection{Acknowledgements}
We thank Eric Vanden-Eijnden and Miguel Onorato for their suggestions and helpful discussions about their work. The first author also thanks Promit Ghosal for some rich conversations about Large Deviations Theory.


\section{Large Deviations Principle for the linear Schr\"odinger equation}\label{sec:linearLDP}

\noindent Consider the linear Schr\"odinger equation on the torus $\T=[0,2\pi]$:
\begin{equation}\label{eq:LSn}
\left\lbrace\begin{array}{l}
i\, \pa_t u  + \Delta u  = 0, \\
u(t,x)|_{t=0}=u_0,\end{array}\right.
\end{equation}
where the random initial data, $u_0$, is given as in \eqref{eq:intro_data}. Our goal is to estimate
\[
\pp{\sup_{x \in \T} |u(t,x)| \geq z(\varepsilon)}
\] 
sharply, for some fixed $t \in \R$ and $z(\varepsilon)$ large. More concretely, we want to obtain a \textbf{Large Deviations Principle (LDP)} for $\sup_{x \in \T} |u(t,x)|$. This requires that we find a lower semi-continuous function, $I: [0, +\infty) \mapsto [0, +\infty]$, and a real number $\alpha >0$, such that, for any $z_0 > 0$ fixed, we have
\begin{equation}\label{eq:LDP_lower}
- \inf_{z \in (z_0, +\infty)} I(z) \le \liminf_{\varepsilon \rightarrow 0^+} \varepsilon^{\alpha} \log\left(\pp{\sup_{x \in \T} |u(t,x)| \geq z_0\varepsilon^{-1/2}}\right),
\end{equation}
and
\begin{equation}\label{eq:LDP_upper}
\limsup_{\varepsilon \rightarrow 0^+} \varepsilon^{\alpha}\log\left(\pp{\sup_{x \in \T} |u(t,x)| \geq z_0\varepsilon^{-1/2}}\right) \le - \inf_{z \in [z_0, +\infty)} I(z).
\end{equation}

\noindent We gather these results in the following proposition.

\begin{prop}\label{prop:linearLDP} Consider the linear Schr\"odinger equation on the torus $\T=[0,2\pi]$ as in \eqref{eq:LSn}, with random initial data given by $u_0=\sum_{k\in\Z} c_k \eta_k e^{ik x}$, presented in \eqref{eq:intro_data}. Then,
\begin{equation}\label{eq:linearLDP}
\lim_{\varepsilon\rightarrow 0^{+}} \varepsilon\, \log \P \left( \norm{e^{-it\Delta} u_0}_{L^{\infty}_x} > z_0 \, \varepsilon^{-1/2} \right) =- \frac{z_0^2}{\sum_{k\in\Z} c_k^2},
\end{equation}
for any $t \in \R$ and $z_0 >0$. In particular, \eqref{eq:LDP_lower} and \eqref{eq:LDP_upper} are satisfied with rate function $I(z) = \frac{z^2}{\sum_{k\in\Z} c_k^2}$ and $\alpha = 1$.
\end{prop}

\begin{rk}\label{rk:chainEqualities}
	A posteriori we see that our choice for $I$ is actually continuous, and hence 
	\[
	\inf_{z \in [z_0, +\infty)} I(z) = \inf_{z \in (z_0, +\infty)} I(z).
	\]
	This implies that \eqref{eq:LDP_lower} and \eqref{eq:LDP_upper} is actually a chain of equalities, and we can prove the stronger result presented in \eqref{eq:linearLDP}.
\end{rk}

\begin{rk}\label{rk:nonGaussian}
It is worth noticing that $I$ is a quadratic function. This corresponds to a sub-Gaussian behavior for the tails of $\sup_{x\in\T} |u(t,x)|$. Nevertheless, it is easy to see that $\sup_{x\in\T} |u(t,x)|$ itself is not Gaussian.
\end{rk}

\begin{rk}
 In the Probability literature, Lindgren studied various path properties of Gaussian fields around its local maxima in a series of papers \cite{Lindgren1, Lindgren2,Lindgren4, Lindgren3}. However, these results do not apply directly to our work since we are interested in establishing an LDP for the global maxima on the space variable.
 \end{rk}

We divide the proof of \Cref{prop:linearLDP} into two parts: first, we establish the lower bound, which follows from the fact that we can explicitly derive the distribution of $u(t,x)$, for a fixed pair $(t,x) \in \R \times \T$; and then, we prove the upper bound estimating from above $\sup_{x\in\T} |u(t,x)|$, and applying the G\"artner-Ellis Theorem.

\subsection{Lower Bound}\label{sec:linearLowerBound} The solution $u$ for the linear flow in \eqref{eq:LSn} can be written as 
\begin{equation}\label{eq:solution}
u(t,x)= e^{-it\Delta} u_0(x) = \sum_{k \in \Z} c_k\eta_k\, e^{i\, k x- i\, k^2 t}, 
\end{equation}
for $(t,x) \in \R \times \T$. Since all the terms in the sum are independent complex normal random variables, so is $u(t,x)$ as long as $\sum_{k \in \Z} \ee{u^k(t,x)}$, $\sum_{k \in \Z} \ee{u^k(t,x)u^k(t,x)}$, and $\sum_{k \in \Z}\ee{u^k(t,x)\overline{u^k(t,x)}}$ converge, where $u^k(t,x):=c_k\eta_k\, e^{i \, k x- i\, k^2 t}$. It is easy to see that
\[
\sum_{k \in \Z} \ee{u^k(t,x)} = \sum_{k \in \Z} c_k\ee{\eta_k}e^{i \, k x- i\, k^2 t} = 0, 
\]
\[
\sum_{k \in \Z} \ee{u^k(t,x)u^k(t,x)} = \sum_{k \in \Z} c_k^2\ee{\eta_k\eta_k}e^{i \, 2k x- i\, 2k^2 t} = 0, 
\]
and
\[
\sum_{k \in \Z}\ee{u^k(t,x)\overline{u^k(t,x)}} = \sum_{k \in \Z} c_k^2\ee{\eta_k \bar{\eta_k}} =  \sum_{k \in \Z} c_k^2 < \infty.
\]

Therefore, for $(t,x) \in \R \times \T$ fixed, $u(t,x)$ is a complex Gaussian distribution with mean $0$ and variance $\sum_{k \in \Z} c_k^2$. It is worth noting that this distribution depends neither on $t$ nor on $x$. That is, as a random field, $(t,x) \mapsto u(t,x)$ is at stationarity.

It is a well-known result in Probability \cite{siddiqui1962some} that the modulus of a complex Gaussian random variable with mean $0$ and variance $2\sigma^2$ follows a \textbf{Rayleigh distribution} with parameter $\sigma > 0$ and probability density function
\begin{equation}\label{eq:densityRayleigh}
f(x \, | \, \sigma) = \frac{x}{\sigma^2}e^{-\frac{x^2}{2\sigma^2}}, \qquad x \ge 0. 
\end{equation}

\noindent This fact allows us to compute the lower bound as follows:
\begin{align*}
& \pp{\sup_{x \in \T} |u(t,x)| \geq z_0 \varepsilon^{-1/2}} \ge \pp{|u(t,0)| \geq z_0 \varepsilon^{-1/2}} \\
& = \int_{z_0 \varepsilon^{-1/2}}^{+\infty}\frac{2x}{\sum_{k \in \Z} c_k^2}\exp\left(-\frac{x^2}{\sum_{k \in \Z} c_k^2}\right)dx = \exp\left(-\frac{z_0^2 \varepsilon^{-1}}{\sum_{k \in \Z} c_k^2}\right).
\end{align*}

\noindent Hence, if we take $\alpha = 1$ and define $I(z):= \frac{z^2}{\sum_{k \in \Z} c_k^2}$, we can establish \eqref{eq:LDP_lower}, as we wanted.

\subsection{Upper bound}\label{sec:linearUpperBound}

Since $\eta_k$ is a standard complex Gaussian, we can write $\eta_k$ as $R_k e^{i \varphi_k}$, where $R_k \stackrel{i.i.d.}{\sim} Rayleigh(1/\sqrt{2})$ and $\varphi_k \stackrel{i.i.d.}{\sim} U[0,2\pi]$, with $R_k$ and $\varphi_k$ independent of each other (see \cite{anderson1958introduction}, Section 2.7).
Then, from \eqref{eq:solution} we have:
\begin{equation}\label{eq:overest}
\begin{aligned}
 \sup_{x \in \T} |u(t,x)|^2 = \sum_{k \in \Z} (c_k R_k)^2 + \sup_{x \in \T} \sum_{j\neq k}^{\ast} c_j c_k R_j R_k \cos(\psi_j -\psi_k) \\
 \leq  \sum_{k \in \Z} (c_k R_k)^2 + \sum_{j\neq k}^{\ast} c_j c_k R_j R_k = \left(\sum_{k \in \Z} c_k R_k\right)^2,
 \end{aligned}
\end{equation}
where $\sum^{\ast}$ indicates a sum in all $j, k \in \Z$, and $\psi_k (t,x)= \varphi_k + k x - k^2 t$. Using \eqref{eq:overest}, we can estimate from above the probability in \eqref{eq:linearLDP} to get
\begin{equation}\label{eq:probUpperBound}
 \P \left( \sup_{x \in \T} |u(t,x)|\geq z_0\varepsilon^{-1/2}\right) \leq \P \left( \sum_{k \in \Z} c_k R_k \geq z_0\varepsilon^{-1/2}\right).
\end{equation}

In order to estimate the right-hand side of \eqref{eq:probUpperBound}, we will derive an LDP for the continuous parameter family $\mu_{\varepsilon}$, defined as the probability measures of 
\begin{equation}\label{eq:probMeasuresGE}
Z_{\varepsilon}:=\varepsilon^{1/2}\sum_{k \in \Z} c_k R_k,
\end{equation}
for $\varepsilon >0$. The advantage of \eqref{eq:probMeasuresGE} is that we have gotten rid of the supremum on $\T$, and recovered the classic formulation for an LDP, where the G\"artner-Ellis Theorem may apply. We present below the version of this result that we will use in our work. It corresponds to a modification of Theorem 2.3.6 in \cite{dembo2011large}.
\begin{thm}[G\"artner-Ellis Theorem]\label{thm:GE}
For $\lambda \in \R$ and $\varepsilon >0$, let us define
\begin{equation}\label{eq:GE-MGF-muepsilon}
    \Lambda_{\varepsilon}(\lambda):= \log \ee{e^{\lambda Z_{\varepsilon}}},
\end{equation}
where $Z_{\varepsilon}$ has distribution $\mu_{\varepsilon}$, as in \eqref{eq:probMeasuresGE}. If the function $\Lambda(\cdot)$, defined as the limit
\begin{equation}\label{eq:GE-MGF-limit}
    \Lambda(\lambda):= \lim_{\varepsilon \rightarrow 0^+}\varepsilon \Lambda_{\varepsilon}(\varepsilon^{-1}\lambda),
\end{equation}
exists for each $\lambda \in \R$, takes values in $\R$, and is differentiable, then we have that
\begin{equation}\label{eq:GE_LDP_lower}
- \inf_{z \in (z_0, +\infty)} \Lambda^*(z) \le \liminf_{\varepsilon \rightarrow 0^+} \varepsilon \log \mu_{\varepsilon}((z_0, +\infty)),
\end{equation}
and
\begin{equation}\label{eq:GE_LDP_upper}
\limsup_{\varepsilon \rightarrow 0^+} \varepsilon \log \mu_{\varepsilon}([z_0, +\infty)), \le - \inf_{z \in [z_0, +\infty)} \Lambda^*(z),
\end{equation}
where $\Lambda^*(\cdot)$ is the Fenchel-Legendre transform of $\Lambda(\cdot)$.
\end{thm}

\begin{rk}
In the Probability literature, the function $\Lambda_{\varepsilon}(\lambda)$ is known as the \textbf{cumulant-generating function} of $Z_{\varepsilon}$.
\end{rk}

\begin{rk}
For the purpose of this section, i.e., to derive an upper bound for the right-hand side of \eqref{eq:probUpperBound}, we will only need \eqref{eq:GE_LDP_upper}. However, we present in \Cref{thm:GE} both the upper and lower bounds, since we will use the latter in \Cref{sec:MLE}.
\end{rk}

Next, let us compute the limit in \eqref{eq:GE-MGF-limit}, and check that the function $\Lambda(\cdot)$ satisfies the assumptions required to apply \Cref{thm:GE}. By direct computation using the density function for $R_k$ in \eqref{eq:densityRayleigh} and the Monotone Convergence Theorem, we obtain
\begin{equation}\label{eq:mainsum_upperb}
\Lambda_{\varepsilon} (\varepsilon^{-1} \, \lambda) = \lim_{n \rightarrow \infty} \sum_{|k| \le n} \log \left( 1+ \sqrt{\pi}\varepsilon^{-1/2}\,\lambda\,c_k\, \exp\left(\frac{\varepsilon^{-1} \, \lambda^2 \, c_k^2}{4}\right) \, \P (X_k\geq 0)\right),
\end{equation}
where $X_k$ are independent normal random variables with mean $\frac{\varepsilon^{-1/2}\, \lambda\,  c_k}{2}$ and variance $c_k^2/2$. Before taking the limit as $\varepsilon \rightarrow 0^+$, we need to study the series in \eqref{eq:mainsum_upperb} and prove that it actually converges.

Fix $\lambda \in \R$ and $\varepsilon >0$, and denote $\varepsilon^{-1/2}\,\lambda$ by $\lambda_{\varepsilon}$. Since all the terms are positive, we can bound $\P (X_k\geq 0)$ by $1$, and define the functions
\begin{equation}\label{eq:def_Fn}
F_{n}(\lambda_{\varepsilon})=\prod_{|k|\leq n} \left( 1+ \sqrt{\pi}\, \lambda_{\varepsilon}\, c_k \,\exp\left(\frac{\lambda_{\varepsilon}^2 \, c_k^2}{4}\right)\right),
\end{equation}
for any $n \in \N$. It is clear that $\Lambda_{\varepsilon} (\varepsilon^{-1} \, \lambda) \le \lim_{n \rightarrow \infty} \log F_{n}(\lambda_{\varepsilon})$. Next, we can rewrite the product in \eqref{eq:def_Fn} as a sum over all subsets $P$ of $\mathcal{P}_n:=[-n,n]\cap \Z$, namely
\[
F_{n}(\lambda_{\varepsilon})= \sum_{P\subset \mathcal{P}_n} e^{V(P, \, \lambda_{\varepsilon})},
\]
where
\begin{equation}\label{eq:V_Pn}
V(P, \, \lambda_{\varepsilon})= \frac{1}{2}\, |P|\, \log( \pi) + |P| \, \log\left(\lambda_{\varepsilon}\right) + \sum_{j \in P} \log(c_j) + \frac{\lambda_{\varepsilon}^2}{4}\sum_{j \in P} c_j^2,
\end{equation}
with the convention that $V(\emptyset,\lambda_{\varepsilon})=0$.

Next we wish to maximize $V(P,\lambda_{\varepsilon})$ over all $P\subset \mathcal{P}_n=[-n,n]\cap \Z$, of which there are finitely many. Let  $P^{\ast}_{n}$ be the subset where this maximum is attained. If we define the function $f$ as $f(r):= \sqrt{\pi} r\,e^{r^2/4}$, it is easy to see that 
\begin{equation}\label{eq:def_function_f}
P^{\ast}_{n}=\mbox{arg max}_{P\subset \mathcal{P}_n} \prod_{j \in P} f\left(\lambda_{\varepsilon}\,c_j\right).
\end{equation}
The function $f$ is monotone increasing in $r$, and the coefficients $c_j$ are symmetric in $j$ and decreasing in $|j|$. As a result, the optimal set $P^{\ast}_{n}$ will be of the form $[-k^{\ast},k^{\ast}]\cap\Z$ for some $k^{\ast}\in\{0,\ldots, n\}$. In fact, we can rewrite the set $P^{\ast}_{n}=P^{\ast}_{n}(\lambda_{\varepsilon})$ as follows:
\begin{equation}\label{eq:optimalset}
P^{\ast}_{n}(\lambda_{\varepsilon}) := \left\{j\in [-n,n]\cap\Z: f\left(\lambda_{\varepsilon} \,c_j\right) > 1\right\} = \left\{j\in [-n,n]\cap\Z: c_j > \mathcal{C}/\lambda_{\varepsilon}\right\},
\end{equation}
where\footnote{We can invert $f$ since it is monotone increasing. The constant $\mathcal{C}$ can be explicitly computed, and it is approximately equal to $0.5264$.} $\mathcal{C}:=f^{-1}(1)$. Note that for all $n$ such that $c_n<\mathcal{C}/\lambda_{\varepsilon}$ the set $P^{\ast}_n(\lambda_{\varepsilon})$ is \emph{independent of $n$}. Therefore, we will assume that $n$ is large enough and write $P^{\ast}(\lambda_{\varepsilon})$ from now on (since our goal is to take the limit as $n \rightarrow \infty$).

As we mentioned before, $P^{\ast}=[-k^{\ast},k^{\ast}]\cap\Z$ where $k^{\ast}$ depends on $\lambda_{\varepsilon}$, but not on $n$. The positive integer $k^{\ast}$ is characterized by the property
\begin{equation}\label{eq:about_kast}
\lambda_{\varepsilon} \in \left( \frac{\mathcal{C}}{c_{k^{\ast}}}, \frac{\mathcal{C}}{c_{k^{\ast}+1}}\right].
\end{equation}
With this in mind, we rewrite $F_n$ in \eqref{eq:def_Fn} as 
\[
F_n (\lambda_{\varepsilon}) = e^{V(P^{\ast},\, \lambda_{\varepsilon})} \, \sum_{P\subset \mathcal{P}_n} e^{V (P, \, \lambda_{\varepsilon})- V(P^{\ast},\, \lambda_{\varepsilon})}.
\]

In order to take the limit as $n\rightarrow \infty$ in the previous expression, and to conclude that $\Lambda_{\varepsilon} (\varepsilon^{-1}\lambda)$ is finite, we must estimate the difference $V (P, \, \lambda_{\varepsilon})- V(P^{\ast},\, \lambda_{\varepsilon})$. Given the construction of $P^{\ast}$, we may always bound it by $0$. Unfortunately, this gives rise to the following naive estimate 
\[
F_n (\lambda_{\varepsilon}) \leq e^{V(P^{\ast},\, \lambda_{\varepsilon})} \, \sum_{P\subset \mathcal{P}_n} 1 =  e^{V(P^{\ast},\, \lambda_{\varepsilon})}\, 2^{2n+1}
\]
which is insufficient. For this reason, a careful analysis of the function
\begin{equation}\label{eq:G_P}
G(P):= V (P, \, \lambda_{\varepsilon})- V(P^{\ast},\, \lambda_{\varepsilon})
\end{equation}
is necessary. We present this result in the following proposition.

\begin{prop}\label{thm:final_estimate_G} Fix $b >0$. Then, for any subset $P \subset [-n,n]\cap\Z$ the following estimate holds
\begin{equation}\label{eq:final_estimate_G}
G(P)\leq - \frac{b}{2} \cdot \sum_{j\in \hat{P}} \big| |j|-k^{\ast}\big|.
\end{equation}
where $\hat{P}:= \left(P \Delta P^*\right) - \left\{\pm \, (k^{\ast} + 1) \right\}$, and $\Delta$ denotes the symmetric difference between two sets.
\end{prop}

\begin{rk}
Throughout the proof of \Cref{thm:final_estimate_G}, and in the rest of the paper, we will assume that $c_j=e^{-b|j|}$ for the sake of conciseness. It is easy to see that one can adapt our arguments for the case $c_j=e^{-b|j|^2}$.

\end{rk}

\begin{proof}
First of all, by \eqref{eq:V_Pn} and \eqref{eq:G_P}, we have
\begin{align}\label{eq:additive_G}
G(P) & = \sum_{j\in P} \log\left( f(\lambda_{\varepsilon} c_j)\right)\ - \sum_{j\in P^{\ast}} \log\left( f(\lambda_{\varepsilon} c_j)\right)\\
&  = \sum_{j\in P - P^{\ast}} \log\left( f(\lambda_{\varepsilon} c_j)\right)\   - \sum_{j\in P^{\ast} - P} \log\left( f(\lambda_{\varepsilon} c_j)\right)\nonumber
\end{align}
 for any $P\subset \mathcal{P}_n$, with $f$ as in \eqref{eq:def_function_f}. We reduce the proof of \eqref{eq:final_estimate_G} to three different cases: when $P = P^* \cup A$, with $A \cap P^* = \emptyset$; when $P = P^* - B$, with $B \subseteq P^*$; and finally, a mixed case when $P = \left(P^* \cup A\right) - B$, with $A$ and $B$ as in the previous cases.
 
 \noindent {\bf Case 1:} Let us start with $P = P^* \cup A$, with $A \cap P^* = \emptyset$. By \eqref{eq:additive_G}, we have that
 \begin{equation}\label{eq:boundGcase1}
G(P) = \sum_{j\in A}\log\left( f(\lambda_{\varepsilon} c_j)\right)=\sum_{j\in A}\left(\frac{1}{2}\log \pi + \log (\lambda_{\varepsilon} \, c_j ) + \frac{\lambda_{\varepsilon}^2 \, c_j^2}{4}\right).
\end{equation}
Let us study each of the terms in the previous expression. By \eqref{eq:about_kast},
\begin{equation}\label{eq:controlG1}
\frac{\mathcal{C}^2}{4}\, e^{-2b (|j|-k^{\ast})}= \frac{\mathcal{C}^2}{4} \, \frac{c_j^2}{c_{k^{\ast}}^2}< \frac{\lambda_{\varepsilon}^2 \, c_j^2}{4} \leq \frac{\mathcal{C}^2}{4} \, \frac{c_j^2}{c_{k^{\ast}+1}^2} = \frac{\mathcal{C}^2}{4}\, e^{-2b (|j|-k^{\ast}-1)}.
\end{equation}
Similarly, we have that 
\begin{align}\label{eq:controlG2}
-b (|j|-k^{\ast})+ \log\mathcal{C}  = \log \left(\frac{\mathcal{C}\, c_j}{c_{k^{\ast}}}\right) \leq  \log \left(\lambda_{\varepsilon} c_j \right) \leq \log \left(\frac{\mathcal{C}\, c_j}{c_{k^{\ast}+1}}\right)
& = -b (|j|-k_{\ast}-1)+ \log\mathcal{C}.
\end{align}

\noindent Applying \eqref{eq:controlG1} and \eqref{eq:controlG2} to \eqref{eq:boundGcase1}, we obtain
\begin{align*}
G(P) & \leq \frac{|A|}{2}\log \pi + \sum_{j\in A} \left(-b (|j|-k^{\ast}-1) + \log\mathcal{C} + \frac{\mathcal{C}^2}{4}\, e^{-2b (|j|-k^{\ast}-1)}\right) \\ & \le |A|\left(\frac{1}{2}\log \pi + \log\mathcal{C} + \frac{\mathcal{C}^2}{4}\right) - \sum_{j\in A} b (|j|-k^{\ast}-1) = - b \sum_{j\in A} (|j|-k^{\ast}-1),
\end{align*}
where the last equality follows by the definition of $\mathcal{C}$. In the second inequality we used that, since $j\in A$ and $A \cap P^* = \emptyset$, then $|j|>k^{\ast}$ and hence, the right-hand side of \eqref{eq:controlG1} is bounded by $\frac{\mathcal{C}^2}{4}$. Using the fact that $|j|>k^{\ast}$ once more, we can rewrite the previous estimate as
\begin{align*}
G(P) \le - b \, \sum_{j\in A}\left(\big| |j|-k^{\ast}\big|-1\right).
\end{align*}

\noindent Finally, noting that
\begin{equation}\label{eq:estimateMax}
\big| |j|-k^{\ast}\big|-1 \ge \frac{1}{2}\big| |j|-k^{\ast}\big|,
\end{equation}
for all $j \in \Z - \left\{\pm\,(k^{\ast} - 1), \pm \, k^{\ast}, \pm \, (k^{\ast} + 1) \right\}$, we obtain
\begin{equation}\label{eq:final_estimate_G_case1}
G(P) \le - \frac{b}{2} \, \sum_{j\in \hat{P}}\big| |j|-k^{\ast}\big|,
\end{equation}
where in this case $\hat{P} = A - \left\{\pm \, (k^{\ast} + 1) \right\}$. This concludes the proof of \eqref{eq:final_estimate_G} for the first case.

\noindent {\bf Case 2:} Now, let us assume that $P = P^* - B$, with $B \subseteq P^*$. Notice that this implies that $|j|\leq k^{\ast}$, for any $j \in B$. In this case, \eqref{eq:additive_G} becomes
 \begin{equation}\label{eq:boundGcase2}
G(P) = -\sum_{j\in B}\log\left( f(\lambda_{\varepsilon} c_j)\right)=-\sum_{j\in B}\left(\frac{1}{2}\log \pi + \log (\lambda_{\varepsilon} \, c_j ) + \frac{\lambda_{\varepsilon}^2 \, c_j^2}{4}\right).
\end{equation}
Using \eqref{eq:controlG1} and \eqref{eq:controlG2} again, we have that
\begin{align*}
G(P) & \leq -\frac{|B|}{2}\log \pi - \sum_{j\in B} \left(-b (|j|-k^{\ast}) + \log\mathcal{C} + \frac{\mathcal{C}^2}{4}\, e^{-2b (|j|-k^{\ast})}\right) \\ & \le -|B|\left(\frac{1}{2}\log \pi + \log\mathcal{C} + \frac{\mathcal{C}^2}{4}\right) - b \sum_{j\in B} \big| |j|-k^{\ast}\big|= - b \sum_{j\in B} \big| |j|-k^{\ast}\big|.
\end{align*}
Here, we have used that $|j|\leq k^{\ast}$, for any $j \in B$, and the definition of $\mathcal{C}$. Noting that for this case $\hat{P} = B$, we can conclude
 \begin{equation}\label{eq:final_estimate_G_case2}
G(P) \le - \frac{b}{2} \sum_{j\in B} \big| |j|-k^{\ast}\big| = - \frac{b}{2} \sum_{j\in \hat{P}} \big| |j|-k^{\ast}\big|,
\end{equation}
as we wanted to prove.

\noindent {\bf Case 3:} Finally, for the mixed case when $P = \left(P^* \cup A\right) - B$, with $A \cap P^* = \emptyset$ and $B \subseteq P^*$, the expression in \eqref{eq:additive_G} tells us that
 \begin{equation}\label{eq:boundGcase3}
G(P) = \sum_{j\in A}\log\left( f(\lambda_{\varepsilon} c_j)\right) - \sum_{j\in B}\log\left( f(\lambda_{\varepsilon} c_j)\right).
\end{equation}
Hence, combining \eqref{eq:boundGcase1} and \eqref{eq:final_estimate_G_case1} with \eqref{eq:boundGcase2} and \eqref{eq:final_estimate_G_case2}, the right-hand side of \eqref{eq:boundGcase3} can be bounded by
 \begin{equation}
G(P) \le - \frac{b}{2} \sum_{j\in A - \left\{\pm \, (k^{\ast} + 1) \right\}} \big| |j|-k^{\ast}\big| - \frac{b}{2} \sum_{j\in B} \big| |j|-k^{\ast}\big| = - \frac{b}{2} \sum_{j\in \hat{P}} \big| |j|-k^{\ast}\big|,
\end{equation}
where in this case $\hat{P} = \left(A \cup B\right) - \left\{\pm \, (k^{\ast} + 1) \right\}$. This concludes the proof of \Cref{thm:final_estimate_G}.
\end{proof}

\noindent Based on this result, it is natural to introduce the quantity
\begin{equation}\label{eq:def_normPn}
\norm{P} := \sum_{j\in \hat{P}} \big| |j|-k^{\ast}\big|,
\end{equation}
for each subset $P \subset [-n,n]\cap\Z$, where $\hat{P}:= \left(P \Delta P^*\right) - \left\{\pm \, (k^{\ast} + 1) \right\}$, with the convention that the sum over the empty set is zero.  Back to our main problem, \Cref{thm:final_estimate_G} allows us to estimate $F_n$ in \eqref{eq:def_Fn} as follows:
\[
F_n (\lambda_{\varepsilon})= e^{V(P^{\ast},\, \lambda_{\varepsilon})} \, \sum_{P\subset \mathcal{P}_n} e^{G(P)} \le e^{V(P^{\ast},\, \lambda_{\varepsilon})} \,\sum_{P\subset \mathcal{P}_n} e^{- \frac{b}{2}\,\norm{P}}.
\]
Hence, our next objective is to show that the latter sum is bounded uniformly on $n$ and $\lambda_{\varepsilon}$, i.e.
\begin{equation}\label{eq:prelimit_n}
\sum_{P\subset \mathcal{P}_n} e^{-\frac{b}{2} \, \norm{P}} \lesssim_b 1
\end{equation}
where the implicit constant is independent of $\lambda_{\varepsilon}$ and $n$. In order to do so, we rewrite this sum as follows:
\begin{equation}\label{eq:total_sum2}
  \sum_{P\subset \mathcal{P}_n} e^{-\frac{b}{2}\, \norm{P}} \leq \sum_{m=0}^{\infty} e^{-\frac{b}{2}\,m}\,   |\{P\subseteq [-n,n]\cap \Z : \norm{P}=m\}|.
\end{equation}

It is easy to see that $0 \le \norm{P} \le (2n+1)\, n$. Therefore, the sum on the right-hand side of \eqref{eq:total_sum2} will always be finite since many level sets are empty. The next step is to estimate how many subsets $ P\subseteq [-n,n]\cap \Z$ satisfy $\norm{P}=m$.

\begin{lem}\label{thm:level_sets} The cardinality of the level sets of $\norm{\cdot}$, as defined in \eqref{eq:def_normPn}, satisfies
\begin{equation}\label{eq:level_sets}
|\{P \subseteq [-n,n]\cap \Z \mid \norm{P}=m\}|  \leq p(m)\, 4^{\sqrt{8m}+1},
\end{equation}
where $p(m)$ is the number of partitions of $m$.
\end{lem}
\begin{proof}
Fix $m\in\N_0$. We will bound the left-hand side of \eqref{eq:level_sets} by counting the iterations of an algorithm whose output will include all the possible subsets $P\subseteq [-n,n]\cap \Z$ such that $\norm{P}=m$. As in the proof of \Cref{thm:final_estimate_G}, it will be convenient to use the characterization of $P$ as $\left(P^* \cup A\right) - B$, with $A \cap P^* = \emptyset$ and $B \subseteq P^*$. Recall that in this case, $\hat{P}$ is given by $\left(A \cup B\right) - \left\{\pm \, (k^{\ast} + 1) \right\}$.
\begin{enumerate}
\item \textbf{Step 1:} We write $m= \sum_{j=1}^{M} a_j$ for integers $a_j\geq 1$, and a certain $M \in \N_0$.
\item \textbf{Step 2:} For each $a_j$ from Step 1, we choose one element $j \in [-n,n]\cap \Z$ among $k^{\ast} -a_j$, $k^{\ast} +a_j$, $-k^{\ast} -a_j$, or $-k^{\ast} +a_j$. If $|j| \le k^*$, we add it to $B$. Otherwise, we add it to $A$.
\item \textbf{Step 3:} We add any elements in $\left\{\pm \, (k^{\ast} + 1) \right\}$ to our set $P$.
\end{enumerate}

Since there are $p(m)$ possible options in Step 1, and for each $a_j$ in Step 2, there are 4 available choices, we have that the number of possible sets $P$ that can be constructed with this algorithm is bounded by $p(m) \, 4^{M+1}$. The last step to conclude \eqref{eq:level_sets} is to estimate $M$. Given Step 2, there can be at most 4 possible repetitions of each $a_j$ in $\sum_{j=1}^{M} a_j$. Therefore,
\begin{equation*}
 m = \sum_{j=1}^{M} a_j\geq 4\cdot \sum_{j=1}^{M/4} j = \frac{M}{2}\, \left(\frac{M}{4}+1\right)\geq \frac{M^2}{8}.
 \end{equation*}
As a consequence, $M\leq \sqrt{8m}$ and the result follows.
\end{proof}

\noindent By \eqref{eq:total_sum2} and \Cref{thm:level_sets}, we have that 
\begin{equation*}
\sum_{P\subset \mathcal{P}_n} e^{-\frac{b}{2}\, \norm{P}} \leq \sum_{m=0}^{\infty} e^{-\frac{b}{2}\,m}\,  p(m)\, 4^{\sqrt{8m} + 1},
\end{equation*}
where $p(m)$ is the number of partitions of $m$. Notice that the right-hand side of the previous expression is independent of $\lambda_{\varepsilon}$ and $n$. By Theorem 6.3 in \cite{Andrews},
\begin{equation*}
p(m)\lesssim \exp \left( \pi \sqrt{\frac{2}{3}} \, \sqrt{m}\right).
\end{equation*} 
Consequently, in order to establish the uniform upper bound in \eqref{eq:prelimit_n}, it is enough to prove that the series
\begin{equation*}
\sum_{m=0}^{\infty} \exp\left(-\frac{b}{2}\,m +\pi \sqrt{\frac{2}{3}} \, \sqrt{m} + \log(4)\left(\sqrt{8m} + 1\right) \right)
\end{equation*}
converges, which follows easily by the root test.

\noindent Finally, we are ready to compute the limit in \eqref{eq:GE-MGF-limit}:
\begin{align*}
& \lim_{\varepsilon \rightarrow 0^{+}} \varepsilon \,\Lambda_{\varepsilon} (\varepsilon^{-1} \, \lambda) \leq  \lim_{\varepsilon\rightarrow 0^{+}} \varepsilon \, \lim_{n\rightarrow \infty} \log F_n(\lambda_{\varepsilon})\\
 & = \lim_{\varepsilon\rightarrow 0^{+}} \varepsilon \, V(P^{\ast}(\lambda_{\varepsilon}),\lambda_{\varepsilon})
 +\lim_{\varepsilon\rightarrow 0^{+}} \varepsilon \lim_{n\rightarrow \infty} \, \log\left(\sum_{P\subset \mathcal{P}_n} e^{G(P)}\right).
\end{align*}
The latter limit is zero given that the sum is uniformly bounded on $\lambda_{\varepsilon}$ and $n$ by \Cref{thm:final_estimate_G} and \Cref{thm:level_sets}. Therefore, we only need to compute the limit that corresponds to $V$. First, note that $|P^{\ast}|=2\,k^{\ast}+1$. Using \eqref{eq:about_kast} and the exponential decay of the $c_k$ we can estimate $k^{\ast}$ in terms of $\lambda_{\varepsilon}$:
\[
\mathcal{C}\, e^{b k^{\ast}} < \lambda_{\varepsilon} \leq \mathcal{C}\, e^{b k^{\ast} + b}\quad  \Longrightarrow \quad b \, k^{\ast} \sim \log (\lambda_{\varepsilon}).
\]
Then, using \eqref{eq:V_Pn}, we have:
\[
\begin{split}
 V(P^{\ast}(\lambda_{\varepsilon}),\lambda_{\varepsilon})-\frac{\lambda_{\varepsilon}^2}{4}\sum_{|j|\leq k^{\ast}} c_j^2 & =(2k^{\ast}+1) \, \log(\lambda_{\varepsilon}) +\frac{2k^{\ast}+1}{2}\, \log( \pi) + \sum_{|j|\leq k^{\ast}} \log(c_j)\\
 & \lesssim_b [\log(\lambda_{\varepsilon})]^2 - \sum_{|j|\leq k^{\ast}} b\,|j| \lesssim_b [\log(\lambda_{\varepsilon})]^2.
 \end{split}
\]
As a consequence, 
\begin{equation}\label{eq:Lambda_upperb}
\lim_{\varepsilon \rightarrow 0^{+}} \varepsilon \,\Lambda_{\varepsilon} (\varepsilon^{-1} \, \lambda) \leq \lim_{\varepsilon\rightarrow 0^{+}} \varepsilon \, V(P^{\ast}(\lambda_{\varepsilon}),\lambda_{\varepsilon})=\lim_{\varepsilon\rightarrow 0^{+}} \varepsilon\, \frac{\lambda_{\varepsilon}^2}{4}\sum_{|j|\leq k^{\ast}} c_j^2 = \frac{\lambda^2}{4}\, \sum_{j\in\Z} c_j^2.
\end{equation}
Here we used that $\lambda_{\varepsilon} = \varepsilon^{-1/2}\lambda$, as well as the fact that $k^{\ast}\sim_b \log (\lambda_{\varepsilon})\rightarrow\infty$ as  $\varepsilon\rightarrow 0^{+}$.

Note that the upper bound for the limit in \eqref{eq:Lambda_upperb}  is due to the fact that we estimated from above $\P (X_k>0)\leq 1$ in \eqref{eq:mainsum_upperb}. Similarly, one can estimate from below $\P (X_k>0)\geq 1/2$, since $X_k$ is a normal random variable with positive mean. This leads to a new definition of $F_n(\lambda_{\varepsilon})$ in \eqref{eq:def_Fn}, given by
\[
F_{n}(\lambda_{\varepsilon})=\prod_{|k|\leq n} \left( 1+ \frac{\sqrt{\pi}}{2}\, \lambda_{\varepsilon}\, c_k \,\exp\left(\frac{\lambda_{\varepsilon}^2 \, c_k^2}{4}\right)\right).
\]
Redefining the function $f$ in \eqref{eq:def_function_f} as $f(r):= \frac{\sqrt{\pi}}{2} r\,e^{r^2/4}$, we can repeat the same steps as above, which yield the lower bound 
\begin{equation}\label{eq:Lambda_lowerb}
\frac{\lambda^2}{4}\, \sum_{j\in\Z} c_j^2 = \lim_{\varepsilon\rightarrow 0^{+}} \varepsilon \, V(P^{\ast}(\lambda_{\varepsilon}),\lambda_{\varepsilon}) = \lim_{\varepsilon\rightarrow 0^{+}} \varepsilon \, \lim_{n\rightarrow \infty} \log F_n(\lambda_{\varepsilon}) \le \lim_{\varepsilon \rightarrow 0^{+}} \varepsilon \,\Lambda_{\varepsilon} (\varepsilon^{-1} \, \lambda). 
\end{equation}

\noindent Since the lower and the upper bounds coincide, we can conclude that
\begin{equation}\label{eq:Lambda_finalValue}
\Lambda(\lambda):= \lim_{\varepsilon \rightarrow 0^+}\varepsilon \Lambda_{\varepsilon}(\varepsilon^{-1}\lambda) = \frac{\lambda^2}{4}\, \sum_{j\in\Z} c_j^2.
\end{equation}
Hence, the function $\Lambda(\cdot)$ is clearly well-defined, takes values in $\R$, and is differentiable, and therefore \Cref{thm:GE} applies. Using the definition of the Fenchel-Legendre transform, it is easy to see that
\begin{equation}\label{eq:Lambda_Star}
\Lambda^*(z):= \sup_{\lambda \in \R}\left(\lambda\,z - \Lambda(\lambda)\right) = \frac{z^2}{\sum_{j\in\Z} c_j^2}.
\end{equation}
Combining \eqref{eq:probUpperBound} and \eqref{eq:GE_LDP_upper}, we can conclude that
\begin{equation*}
 \limsup_{\varepsilon \rightarrow 0^+} \varepsilon \log \P \left( \sup_{x \in \T} |u(t,x)|\geq z_0\varepsilon^{-1/2}\right) \leq \limsup_{\varepsilon \rightarrow 0^+} \varepsilon \log \P \left( \sum_{k \in \Z} c_k R_k \geq z_0\varepsilon^{-1/2}\right) \le - \frac{z_0^2}{\sum_{j\in\Z} c_j^2}.
\end{equation*}
This establishes \eqref{eq:LDP_upper} for $\alpha = 1$ and $I(z) = \frac{z^2}{\sum_{k\in\Z} c_k^2}$, and concludes the proof of \Cref{prop:linearLDP}.

\section{Large Deviations Principle for NLS: subcritical times}
\label{sec:subcritical}

\subsection{The linear approximation}

Consider the NLS equation with a cubic nonlinearity:
\begin{equation}\label{eq:cubicNLS}
\begin{cases}
i\pa_t u + \Delta u = \varepsilon^2 |u|^2\, u \\
u(0,x)=\sum_{k\in\Z} c_k \eta_k e^{ik x} 
\end{cases}
\end{equation}

In order to study this equation, we write it as a system of equations for the Fourier coefficients of the solution $u$. Let $v(t,x):= e^{it\Delta} u(t,x)$ and write
\begin{equation}
v(t,x)=\sum_{k\in\Z} v_k (t)\, e^{ikx}
\end{equation}
so that $v_k(t)$ satisfies
\begin{equation}\label{eq:NLS_Fcoeff}
\begin{cases}
i\partial_t v_k = \varepsilon^2 \, \sum_{k=k_1-k_2+k_3} v_{k_1} \overline{v_{k_2}} v_{k_3}\, e^{-it\Omega}\\
v_k(0)=c_k \eta_k
\end{cases}
\end{equation}
where 
\begin{equation}\label{eq:def_omega}
\Omega = |k_1|^2 - |k_2|^2 + |k_3|^2 -|k|^2.
\end{equation}

We introduce the Fourier-Lebesgue spaces $\Fc L^{s,p}=\Fc L^{s,p} (\T)$ of functions $f:\T \rightarrow\C$ with Fourier coefficients $(f_k)_{k\in\Z}$ given by the norm
\begin{equation}\label{eq:def_FL}
\norm{f}_{\Fc L^{s,p}} = \norm{ \langle k\rangle^{s} f_k}_{\ell_k^p}.
\end{equation}
In this section we will restrict ourselves to the space $\Fc L^{0,1}$.

It will be convenient to have explicit control of the solution to \eqref{eq:cubicNLS} in the space $\Fc L^{0,1}$ for small times. To that end, we have the following local-wellposedness result:

\begin{prop}\label{thm:LWP1}
Let $u_0\in \Fc L^{0,1}$ and 
\begin{equation}\label{eq:LWP_time}
T_{\varepsilon} \sim \varepsilon^{-2} \, \norm{u_0}_{\Fc L^{0,1}}^{-2}.
\end{equation}
For times $0\leq t\leq T_{\varepsilon}$, there is a unique solution to the IVP given by \eqref{eq:cubicNLS} which lives in $\Fc L^{0,1}$. Moreover, there exists some positive constant $C$ (independent of $\varepsilon$, $t$ and $u_0$) such that for all times $0\leq t\leq T_{\varepsilon}$ the following inequality holds:
\begin{equation}\label{eq:bootstrap}
\norm{u(t)-e^{-it\Delta} u_0}_{\Fc L^{0,1}} \leq C\, \varepsilon^{2} \, t\, \left( \norm{u-e^{-it\Delta} u_0}_{L^{\infty}([0,t],\Fc L^{0,1})}^3 + \norm{u_0}_{\Fc L^{0,1}}^3 \right)\ .
\end{equation}
\end{prop}
\begin{proof}
Using \eqref{eq:NLS_Fcoeff}, we have that 
\begin{equation}\label{eq:integral_eq}
v_k (t) = v_k(0) -i\, \varepsilon^{2}  \, \int_{0}^{t} \sum_{k=k_1-k_2+k_3} (v_{k_1} \overline{v_{k_2}} v_{k_3})(s)\, e^{-is\Omega} \, ds
\end{equation}
and therefore 
\begin{equation}
\begin{split}
\norm{v(t)}_{\Fc L^{0,1}} \leq & \norm{v(0)}_{\Fc L^{0,1}} + \varepsilon^{2} \, \sum_{k\in\Z} \int_0^t \Big | \sum_{k=k_1-k_2+k_3} (v_{k_1} \overline{v_{k_2}} v_{k_3})(s)\Big | \, ds \\
\leq & \norm{v(0)}_{\Fc L^{0,1}} + \varepsilon^{2} \,  \sum_{k_1,k_2,k_3} \int_0^t \Big |(v_{k_1} \overline{v_{k_2}} v_{k_3})(s)\Big | \, ds \\
\leq & \norm{v(0)}_{\Fc L^{0,1}} + \varepsilon^{2} \,   \int_0^t \norm{v(s)}_{\Fc L^{0,1}}^3\, ds .
\end{split}
\end{equation}
Consequently,
\begin{equation}\label{eq:fixed_point}
\sup_{0\leq t\leq T} \norm{v(t)}_{\Fc L^{0,1}} \leq \norm{v(0)}_{\Fc L^{0,1}} + \varepsilon^{2} \,  T\,  \sup_{0\leq t\leq T} \norm{v(t)}_{\Fc L^{0,1}}^3. 
\end{equation}
Similarly
\begin{equation}\label{eq:fixed_point1}
\sup_{0\leq t\leq T} \norm{v(t)-w(t)}_{\Fc L^{0,1}} \leq  \varepsilon^{2} \,  T\,  \sup_{0\leq t\leq T} \norm{v(t)-w(t)}_{\Fc L^{0,1}}(\norm{v(t)}_{\Fc L^{0,1}}^2+\norm{w(t)}_{\Fc L^{0,1}}^2). 
\end{equation}
A standard argument using a contraction mapping  theorem yields existence, uniqueness and the local time of existence \eqref{eq:LWP_time}. Note that $\norm{u}_{\Fc L^{0,1}}= \norm{v}_{\Fc L^{0,1}}$ so all these results apply to $u(t,x)$.

Finally, we can repeat the argument starting with \eqref{eq:integral_eq} to find
\begin{equation}
\begin{split}
\norm{v(t)-v(0)}_{\Fc L^{0,1}} \leq &\ \varepsilon^{2} \,  t\, \norm{v}_{L^{\infty}([0,t],\Fc L^{0,1})}^3 \\
\leq  &\ C\, \varepsilon^{2} \,  t\,\left( \norm{v-v(0)}_{L^{\infty}([0,t],\Fc L^{0,1})}^3 + \norm{v(0)}_{\Fc L^{0,1}}^3\right).
\end{split}
\end{equation}
We use the fact that $v=e^{it\Delta} u$ and that $\norm{e^{it\Delta} f}_{\Fc L^{0,1}} = \norm{f}_{\Fc L^{0,1}}$ to finish the proof of \eqref{eq:bootstrap}.
\end{proof}

\begin{rk}\label{rk:GWP}

We note that in our case $u_0 \in \cap_{s\geq 0} H^s(\T)$ almost surely. Because of this, the solution in \Cref{thm:LWP1} is actually global. In fact, 
the NLS initial value problem \eqref{eq:cubicNLS} is globally well-posed in $H^s, \, s\geq 0$, \cite{Bourgain-lwp}. This is a consequence of the fact that via the $L^4$ Strichartz estimates one can prove local well-posedness in a small interval of time that is inversely proportional to the mass of the initial data, and then iterate  using the conservation of mass itself.  
Our initial data a.s. lives there, so the solution $u(t)$ to \eqref{eq:cubicNLS}  is well-defined and exists for all times. Moreover, $H^1 (\T ) \subset \Fc L^{0,1}(\T)$ thanks to the Cauchy-Schwarz inequality:
\[ 
\norm{u}_{\Fc L^{0,1}} = \sum_{k} |u_k| \leq \left(\sum_{k} \langle k\rangle^2 |u_k|^2\right)^{1/2} \cdot \left(\sum_{k} \frac{1}{\langle k\rangle^2}\right)^{1/2} \sim \norm{u}_{H^1}.
\]
\end{rk}

We want to use \Cref{thm:LWP1} to show that as long as our initial datum is not too large, the difference between the solution to the nonlinear equation and the linear approximation remains smaller than the initial datum itself. More precisely, we have that 

\begin{cor}\label{thm:small_error} Fix $\gamma\in (0,1]$, $c_1,c_2>0$ and 
\begin{equation}\label{eq:cond_delta}
    0<\delta < \frac{\gamma}{4}.
\end{equation}
 Then there exists $\varepsilon_0 = \varepsilon_0 (\delta,\gamma,c_1,c_2)$ such that for all $\varepsilon<\varepsilon_0$ the following holds: if $T=c_1\,\varepsilon^{-1+\gamma}$ and $\norm{u_0}_{\Fc L^{0,1}} < c_2\, \varepsilon^{-1/2-\delta}$ then 
\[ 
\norm{u-e^{-it\Delta} u_0}_{L^{\infty}([0,T],\Fc L^{0,1})} < \varepsilon^{-1/2+\delta} .
\]
\end{cor}
\begin{proof}
Suppose that $\norm{u_0}_{\Fc L^{0,1}} < c_2\, \varepsilon^{-1/2-\delta}$. Then \eqref{eq:LWP_time} and \eqref{eq:cond_delta} guarantee that 
\[
T_{\varepsilon} \geq c_2^{-2} \varepsilon^{-1+2\delta} \gg c_1 \varepsilon^{-1+\gamma}=T
\]
if $\varepsilon<\varepsilon_0$ is small enough.

Let us further reduce $\varepsilon_0$ (if necessary) so that the following condition is true for all $\varepsilon<\varepsilon_0$: 
\begin{equation}\label{eq:cond_epsilon0}
C\, c_1\, \left( \varepsilon^{-1/2+\gamma+3\delta} + c_2^3\,\varepsilon^{-1/2+\gamma-3\delta} \right) \leq \frac{1}{2} \,\varepsilon^{-1/2+\delta},
 \end{equation}
where $C$ is the positive constant from \eqref{eq:bootstrap}. Note that this is possible thanks to \eqref{eq:cond_delta}.

For $\varepsilon\leq \varepsilon_0$, let us define
\[ \tau:=\sup\left\lbrace t \in [0,T] \mid \norm{u-e^{-is\Delta} u_0}_{L^{\infty}([0,t],\Fc L^{0,1})} < \varepsilon^{-1/2+\delta}\right\rbrace. \]
We would like to show that $\tau=T$ thanks to a bootstrap argument. For $t\leq \tau$, \eqref{eq:bootstrap} yields:
\begin{equation}
\begin{split}
\norm{u(t)-e^{-it\Delta} u_0}_{\Fc L^{0,1}} \leq &\ C\, \varepsilon^{2} \, T\, \left( \varepsilon^{-3/2+3\delta} + c_2^3\,\varepsilon^{-3/2-3\delta} \right) \\
\leq &\ C\, c_1\,\varepsilon^{1+\gamma}\, \left( \varepsilon^{-3/2+3\delta} + c_2^3\,\varepsilon^{-3/2-3\delta} \right) \\
\leq &\ C\, c_1\, \left( \varepsilon^{-1/2+\gamma+3\delta} + c_2^3\,\varepsilon^{-1/2+\gamma-3\delta} \right) .
\end{split}
\end{equation}
This together with \eqref{eq:cond_epsilon0} shows that $\tau=T$.
\end{proof}

\subsection{Proof of \Cref{thm:mainLDP} for subcritical times}

Using these tools, we are ready to give a proof of the LDP for subcritical times $t\ll \varepsilon^{-1}$, which is essentially an extension of the LDP for the linear flow.

Let us fix a time $t=c_1 \varepsilon^{-1+\gamma}\ll \varepsilon^{-1}$ for some $\gamma\in(0,1]$ and $z_0>0$. We want to study the limit
\begin{equation}\label{eq:subcriticalLDP}
\lim_{\varepsilon\rightarrow 0^{+}} \varepsilon \, \log \P \left( \sup_{x\in\T} |u(t,x)| > z_0\, \varepsilon^{-1/2} \right).
\end{equation}

Let $\mathcal{D}_{\varepsilon}$ denote the event $\left\{\sup_{x\in\T} |u(t,x)| > z_0\, \varepsilon^{-1/2}\right\}$. Fix $\delta$ satisfying \eqref{eq:cond_delta}, and let us consider the set $\mathcal{A}_{\varepsilon}$ of initial data $u_0$ such that $\norm{u_0}_{\Fc L^{0,1}}> \varepsilon^{-1/2-\delta}$. Note that $t\gg T_{\varepsilon}$ in \eqref{eq:LWP_time} in this case, but $\norm{u(t)}_{L^{\infty}_x}$ is still a.s. finite thanks to \Cref{rk:GWP}.

First we derive an upper bound. By the triangle inequality and the embedding $\Fc L^{0,1}\subset L^{\infty}$ we have that $\mathcal{D}_{\varepsilon}\subset \mathcal{D}_{\varepsilon}^+$ where 
\[
\mathcal{D}^+_{\varepsilon}:=\left\lbrace \norm{e^{-it\Delta} u_0}_{L^{\infty}_x} + \norm{u(t)-e^{-it\Delta} u_0}_{\Fc L^{0,1}} > z_0\,\varepsilon^{-1/2} \right\rbrace .
\]
Finally, let us define the event 
\[
\mathcal{B}_{\varepsilon} := \left\lbrace \norm{u(t)-e^{-it\Delta} u_0}_{\Fc L^{0,1}}\leq z_0\, \varepsilon^{-1/2+\delta}\right\rbrace.
\]
Then we have that 
\begin{equation}\label{eq:subcritical_upperb}
\P (\mathcal{D}_{\varepsilon}) \leq \P (\mathcal{D}^+_{\varepsilon} \cap \mathcal{B}_{\varepsilon}) + \P (\mathcal{D}^+_{\varepsilon} \cap \mathcal{B}_{\varepsilon}^c)
\end{equation}

Let us study the first term in \eqref{eq:subcritical_upperb}. We have that
\[
\P (\mathcal{D}^+_{\varepsilon} \cap \mathcal{B}_{\varepsilon}) \leq \P \left( \norm{e^{-it\Delta} u_0}_{L^{\infty}_x} > z_0 (\varepsilon^{-1/2}-\varepsilon^{-1/2+\delta})\right).
\]
By \Cref{prop:linearLDP}, 
\[
\lim_{\varepsilon\rightarrow 0^{+}} \varepsilon \, \log \P \left( \norm{e^{-it\Delta} u_0}_{L^{\infty}_x} > z_0 (\varepsilon^{-1/2}-\varepsilon^{-1/2+\delta})\right) = -\frac{z_0^2}{\sum_{k\in\Z}c_k^2} .
\]

Finally, we consider the second term in \eqref{eq:subcritical_upperb}. By \Cref{thm:small_error}, if $\varepsilon<\varepsilon_0$ is small enough we can arrange:
\[
\P (\mathcal{D}^+_{\varepsilon} \cap \mathcal{B}_{\varepsilon}^c) \leq \P ( \mathcal{B}_{\varepsilon}^c) \leq \P \left( \norm{u_0}_{\Fc L^{0,1}} \geq  \varepsilon^{-1/2-\delta}\right) = \P (\mathcal{A}_{\varepsilon}).
\]
Note that we can write $\norm{u_0}_{\Fc L^{0,1}}=\sum_{k\in\Z} c_k R_k$ as we did in \eqref{eq:probMeasuresGE}. By \Cref{thm:GE} and \eqref{eq:GE_LDP_upper}, we have that $\log \P (\mathcal{A}_{\varepsilon})\sim -\varepsilon^{-1-2\delta}$, and therefore the second term in the right-hand side of \eqref{eq:subcritical_upperb} is much smaller than the first one for $\varepsilon$ small enough. This concludes the proof of the upper bound.

Next let us prove a lower bound. First of all note that $\mathcal{D}^-_{\varepsilon}\subset \mathcal{D}_{\varepsilon}$ if we define
\[
\mathcal{D}^-_{\varepsilon}:=\left\lbrace  \norm{e^{-it\Delta} u_0}_{L^{\infty}_x}-\norm{u(t)-e^{-it\Delta} u_0}_{\Fc L^{0,1}} > z_0\,\varepsilon^{-1/2} \right\rbrace .
\]

Therefore, we have that 
\begin{equation}\label{eq:subcritical_lowerb}
\begin{split}
\P (\mathcal{D}_{\varepsilon}) & \geq \P (\mathcal{D}^-_{\varepsilon} \cap \mathcal{B}_{\varepsilon})\geq \P \left( \left\lbrace \norm{e^{-it\Delta} u_0}_{L^{\infty}_x}> z_0 (\varepsilon^{-1/2} + \varepsilon^{-1/2+\delta})\right\rbrace \cap \mathcal{B}_{\varepsilon}\right) \\
& \geq \P \left( \norm{e^{-it\Delta} u_0}_{L^{\infty}_x}> z_0 (\varepsilon^{-1/2} + \varepsilon^{-1/2+\delta})\right) - \P ( \mathcal{B}_{\varepsilon}^c)
\end{split}
\end{equation}

By \Cref{prop:linearLDP}, we have that 
\[
\lim_{\varepsilon\rightarrow 0^{+}} \varepsilon \, \log \P \left( \norm{e^{-it\Delta} u_0}_{L^{\infty}_x} > z_0 (\varepsilon^{-1/2}+\varepsilon^{-1/2+\delta})\right) = -\frac{z_0^2}{\sum_{k\in\Z}c_k^2} .
\]
This describes the asymptotic behavior of the first summand in \eqref{eq:subcritical_lowerb}.
By \Cref{thm:small_error}, if $\varepsilon<\varepsilon_0$ the second summand satisfies
\[
\P ( \mathcal{B}_{\varepsilon}^c)\leq \P \left( \norm{u_0}_{\Fc L^{0,1}} \geq  \varepsilon^{-1/2-\delta}\right)=\P (\mathcal{A}_{\varepsilon})
\]
which is a lower order term by the same argument as in the upper bound. This concludes the proof of the theorem.

\section{LPD for NLS: critical times}
\label{sec:critical}

\subsection{The resonant approximation}\label{subsec:resonant}

In order to prove \Cref{thm:mainLDP} for times of order $\O (\varepsilon^{-1})$, the linear approximation is not enough and we must introduce an approximation that better captures the dynamics of the nonlinear equation over such times. This can be achieved by the \emph{resonant approximation}. 

We start by writing the solution to \eqref{eq:cubicNLS} in Fourier series:
\begin{equation}
u(t,x)= \sum_{k\in\Z} u_k(t) \, e^{ikx} \ .
\end{equation}
Then we derive equations for the Fourier coefficients:
\begin{equation}
i \pa_t u_k -|k|^2 u_k = \varepsilon^2 \, \sum_{k=k_1-k_2+k_3} u_{k_1} \overline{u_{k_2}} u_{k_3}.
\end{equation}
As in the linear case, we let $v(t,x):= e^{it\Delta} u(t,x)$ so that 
\begin{equation}\label{eq:pre_resonant}
i \pa_t v_k = \varepsilon^2 \, \sum_{k=k_1-k_2+k_3} v_{k_1} \overline{v_{k_2}} v_{k_3} e^{-it\Omega}
\end{equation}
for $\Omega$ as in \eqref{eq:def_omega}. We would like to approximate this system by the resonant system, which consists of removing any terms on the right-hand side of \eqref{eq:pre_resonant} for which $\Omega\neq0$.

By imposing $k=k_1-k_2+k_3$ and $\Omega=0$ one easily sees that we must have 
\[ k_1=k\ \mbox{and}\ k_2=k_3,\quad \mbox{or}\quad k_3=k\ \mbox{and}\ k_1=k_2. \]
We can remove most of those terms by setting $w(t,x)=e^{-2i t \varepsilon^2 \, M} v(t,x)$ where 
\begin{equation}
M=\sum_{k\in\Z} |v_k (t)|^2 =\sum_{k\in\Z} |u_k (t)|^2
\end{equation}
which is conserved. The remaining system is 
\begin{equation}\label{eq:pre_resonant_Fourier}
i \pa_t w_k = -\varepsilon^2\, |w_k|^2 w_k+ \varepsilon^2 \, \sum_{k=k_1-k_2+k_3,\Omega\neq 0} w_{k_1} \overline{w_{k_2}} w_{k_3} e^{-it\Omega}.
\end{equation}
The resonant approximation is given by $a(t,x) =\sum_{k\in\Z} a_k(t) \, e^{ikx}$ such that 
\begin{equation}\label{eq:resonant_Fourier}
\begin{cases}
i \pa_t a_k = -\varepsilon^2\, |a_k|^2 a_k,\\
a_k (0)= c_k \eta_k .
\end{cases}
\end{equation}
It turns out that one can explicitly solve this system after noting that $\pa_t |a_k(t)|^2 =0$. As a consequence $|a_k(t)|^2=c_k^2 |\eta_k|^2$ for all times $t$ and the system becomes linear. Direct integration then yields:
\begin{equation}\label{eq:def_a}
a(t,x)=\sum_{k\in\Z} c_k \eta_k \, e^{ikx +i t \varepsilon^2 \, c_k^2 |\eta_k|^2}.
\end{equation}
As a final step, one can undo the transformations we previously did and obtain
\begin{equation}\label{eq:resonant_app}
u_{\mathrm{app}} (t,x) = e^{2it \varepsilon^2 \, M}\,  \sum_{k\in\Z} c_k \eta_k \, e^{ikx +i \varepsilon^2 t\, c_k^2 |\eta_k|^2- it k^2}.
\end{equation}
Thanks to the rapid decay of the $c_k$, it is easy to see that $u_{\mathrm{app}}$ is well-defined globally in $t$ almost surely.

Our first result guarantees that the $u_{\mathrm{app}}$ enjoys the same LDP as the linear flow:

\begin{prop}\label{thm:resonantLDP} Fix $z_0>0$. For $t>0$ (which may depend on $\varepsilon$), we have that
\begin{equation}\label{eq:resonantLDP}
\lim_{\varepsilon \rightarrow0^{+}} \varepsilon\, \log \P \left( \sup_{x\in\T} |u_{\mathrm{app}} (t,x)| > z_0\,\varepsilon^{-1/2}\right) = - \frac{z_0^2}{\sum_{k\in\Z} c_k^2}.
\end{equation}
\end{prop}

Before presenting the proof of \Cref{thm:resonantLDP}, we need to establish a lemma about complex Gaussian random variables. In the proof of \Cref{prop:linearLDP}, we used that the solution to the linear equation is given by an infinite sum of complex Gaussian random variables to establish the LDP. For the resonant approximation, it is not clear a priori that this still holds, since in \eqref{eq:resonant_app} the complex exponential is no longer deterministic. \Cref{thm:GaussianInvariance} shows that this transformation actually preserves the distribution of $\eta_k$.

\begin{lem}\label{thm:GaussianInvariance} Let $\eta$ be a standard complex Gaussian and $a \in \R$. Then, $\eta \, e^{i a \,|\eta|^2}$ has again a standard complex Gaussian distribution.
\end{lem}

\begin{proof}
Let us recall that, since $\eta$ is a standard complex Gaussian, we can write $\eta$ as $R e^{i \varphi}$, where $R \sim Rayleigh(1/\sqrt{2})$ and $\varphi \sim U[0,2\pi]$, with $R$ and $\varphi$ independent of each other (see \cite{anderson1958introduction}, Section 2.7). Using this characterization, we have that 
\[
\eta \, e^{i a \,|\eta|^2} = R\, e^{i \varphi + i a \,R^2} = R\, e^{i \widetilde{\varphi}(R)},
\]
where $\widetilde{\varphi}(R):= \varphi + a \,R^2 \left(\text{mod } 2\pi\right)$. Hence, it suffices to show that $\widetilde{\varphi}(R) \sim U[0,2\pi]$, and that $\widetilde{\varphi}(R)$ and $R$ are independent to conclude the desired result.

To show the first statement, let us compute the characteristic function of $\widetilde{\varphi}(R)$: for $\xi \in \R$, we have
\[
\ee{e^{i \xi \widetilde{\varphi}(R)}} = \ee{\ee{e^{i \xi \widetilde{\varphi}(R)}\,\big|\, R}} = \ee{h(R)},
\]
where $h(r) = \ee{e^{i \xi \widetilde{\varphi}(r)}}$, since $R$ and $\varphi$ are independent. Now, using the translation invariance of the uniform distribution, it is easy to see that $\widetilde{\varphi}(r) \sim U[0,2\pi]$, for any $r \in \R$, and hence, $h(r)$ is actually constant and equal to $\ee{e^{i \xi \varphi}}$. This gives us that $\ee{e^{i \xi \widetilde{\varphi}(R)}} = \ee{e^{i \xi \varphi}}$, which yields the desired result, since the characteristic function determines the distribution.

To show that $\widetilde{\varphi}(R)$ and $R$ are independent, we will also use the characteristic function method: for $\xi_1,\xi_2 \in \R$, we have
\begin{align*}
\ee{e^{i \xi_1 \widetilde{\varphi}(R) + i \xi_2 R}} & = \ee{\ee{e^{i \xi_1 \widetilde{\varphi}(R) + i \xi_2 R}\,\big|\, R}} = \ee{\ee{e^{i \xi_1 \widetilde{\varphi}(R)}\,\big|\, R}e^{i \xi_2 R}} \\
& = \ee{\ee{e^{i \xi_1 \varphi}} e^{i \xi_2 R}} = \ee{e^{i \xi_1 \varphi}}\ee{e^{i \xi_2 R}},
\end{align*}
where in the second equality we have used the ``take out what is known'' property of conditional expectation. This implies that $\widetilde{\varphi}(R)$ and $R$ are independent, since its joint characteristic function factorizes. This completes the proof of \Cref{thm:GaussianInvariance}.
\end{proof}

\begin{proof}[Proof of \Cref{thm:resonantLDP}]
As in the proof of \Cref{prop:linearLDP}, we will compute the limit on the left-hand side of  \eqref{eq:resonantLDP} by deriving upper and lower bounds, and checking that they are equal.

Let us start with the upper bound. As in the proof of \Cref{thm:GaussianInvariance}, we can rewrite $\eta_k$ as $R_k e^{i \varphi_k}$. Using the same strategy as in \eqref{eq:overest}, we have
\begin{align*}
 \sup_{x \in \T} |u_{\mathrm{app}}(t,x)|^2 = \sum_{k \in \Z} (c_k R_k)^2 + \sup_{x \in \T} \sum_{j\neq k}^{\ast} c_j c_k R_j R_k \cos(\psi_j -\psi_k) \\
 \leq  \sum_{k \in \Z} (c_k R_k)^2 + \sum_{j\neq k}^{\ast} c_j c_k R_j R_k = \left(\sum_{k \in \Z} c_k R_k\right)^2,
 \end{align*}
where $\sum^{\ast}$ indicates a sum in all $j, k \in \Z$, and $\psi_k (t,x)= \varphi_k + k x + \varepsilon^2 t\, c_k^2 R_k^2 - k^2 t$.  Using this bound, we can estimate from above the probability in \eqref{eq:resonantLDP} to get
\begin{equation}\label{eq:probUpperBound_resonant}
 \P \left( \sup_{x \in \T} |u_{\mathrm{app}}(t,x)|\geq z_0\varepsilon^{-1/2}\right) \leq \P \left( \sum_{k \in \Z} c_k R_k \geq z_0\varepsilon^{-1/2}\right).
\end{equation}
Notice that the right-hand side of \eqref{eq:probUpperBound_resonant} is the same as the one in \eqref{eq:probUpperBound}. Hence, as we did in \Cref{sec:linearUpperBound}, we can apply the G\"artner-Ellis theorem to conclude that
\begin{equation}\label{eq:probUpperBound_limit_resonant}
\begin{aligned}
 & \limsup_{\varepsilon \rightarrow 0^+} \varepsilon \log \P \left( \sup_{x \in \T} |u_{\mathrm{app}}(t,x)|\geq z_0\varepsilon^{-1/2}\right) \\
 & \qquad \leq \limsup_{\varepsilon \rightarrow 0^+} \varepsilon \log \P \left( \sum_{k \in \Z} c_k R_k \geq z_0\varepsilon^{-1/2}\right) \le - \frac{z_0^2}{\sum_{j\in\Z} c_j^2}.
 \end{aligned}
\end{equation}

For the lower bound, we start by noting that, thanks to \eqref{eq:resonant_app} and \Cref{thm:GaussianInvariance}, for each $(t,x) \in \R \times \T$ fixed, $u_{\mathrm{app}}(t,x)$ has the same distribution as the solution to the linear equation in \eqref{eq:LSn}, $e^{-it\Delta} u_0(x)$. That is, $u_{\mathrm{app}}(t,x)$ is equal to the sum of independent complex Gaussian random variables $u^k_{\mathrm{app}}(t,x)$, for $k \in \Z$, with 
\[
\ee{u^k_{\mathrm{app}}(t,x)} = \ee{u^k_{\mathrm{app}}(t,x)u^k_{\mathrm{app}}(t,x)} = 0 \quad \text{and} \quad \ee{u^k_{\mathrm{app}}(t,x)\overline{u^k_{\mathrm{app}}(t,x)}} = c^2_k,
\]
and hence, $u_{\mathrm{app}}(t,x)$ has a complex normal distribution with
\[
\ee{u_{\mathrm{app}}(t,x)} = \ee{u_{\mathrm{app}}(t,x)u_{\mathrm{app}}(t,x)} = 0 \quad \mbox{and}\quad \ee{u_{\mathrm{app}}(t,x)\overline{u_{\mathrm{app}}(t,x)}} = \sum_{k \in \Z}c^2_k.
\] 
Then, using the same arguments as in \Cref{sec:linearLowerBound}, we obtain
\begin{equation}\label{eq:probLowerBound_limit_resonant}
\begin{aligned}
& \liminf_{\varepsilon \rightarrow 0^+} \varepsilon \log \pp{\sup_{x \in \T} |u_{\mathrm{app}}(t,x)| \geq z_0 \varepsilon^{-1/2}} \\
& \qquad \ge \liminf_{\varepsilon \rightarrow 0^+} \varepsilon \log\pp{|u_{\mathrm{app}}(t,0)| \geq z_0 \varepsilon^{-1/2}}
 =  -\frac{z_0^2}{\sum_{k \in \Z} c_k^2}.
 \end{aligned}
\end{equation}

Finally, combining \eqref{eq:probUpperBound_limit_resonant} and \eqref{eq:probLowerBound_limit_resonant}, the expression for the limit in \eqref{eq:resonantLDP} holds, which concludes the proof of \Cref{thm:resonantLDP}.
\end{proof}

\begin{rk}
Note that \eqref{eq:resonantLDP} holds even if $t=\O (\varepsilon^{-1})$, which is the timescale we are interested in: for the upper bound, \eqref{eq:probUpperBound_resonant} provides a uniform bound on time and space; and for the lower bound, $u_{\mathrm{app}}(t,x)$ is at stationarity and hence, it has the same distribution for all $(t,x) \in \R \times \T$.
\end{rk}

\subsection{Bounds on the error}

Our next goal is to obtain good estimates on the error when we approximate the solution $u$ to \eqref{eq:cubicNLS} by the resonant approximation \eqref{eq:resonant_app}. To do so, we will work in the Fourier-Lebesgue space $\Fc L^{2,1}$ defined in \eqref{eq:def_FL}, which is given by the norm:
\begin{equation}
\norm{f}_{\Fc L^{2,1}} = \sum_{k\in\Z} \langle k \rangle^2 |f_k|.
\end{equation}

Our first result is an analogue to \Cref{thm:LWP1} in this new space:

\begin{prop}\label{thm:LWP2}
If $u_0\in \Fc L^{2,1}$, then there exists $T_{\varepsilon}>0$ such that the unique solution to the IVP \eqref{eq:cubicNLS} lives in $\Fc L^{2,1}$ for all times $0\leq t\leq T_{\varepsilon}$. Moreover
\begin{equation}\label{eq:LWP_time2}
T_{\varepsilon} \sim \varepsilon^{-2} \, \norm{u_0}_{\Fc L^{2,1}}^{-2}.
\end{equation}
\end{prop}
\begin{proof}
Using \eqref{eq:pre_resonant}, we have that 
\begin{equation}
v_k (t) = v_k(0) -i \varepsilon^2  \, \int_{0}^{t} \sum_{k=k_1-k_2+k_3} (v_{k_1} \overline{v_{k_2}} v_{k_3})(s)\, e^{-is\Omega} \, ds
\end{equation}
and therefore 
\begin{equation}
\begin{split}
\norm{v(t)}_{\Fc L^{2,1}} \leq & \norm{v(0)}_{\Fc L^{2,1}} + \varepsilon^{2} \, \sum_{k\in\Z} \langle k\rangle^2 \int_0^{t} \Big | \sum_{k=k_1-k_2+k_3} (v_{k_1} \overline{v_{k_2}} v_{k_3})(s)\Big | \, ds \\
\leq & \norm{v(0)}_{\Fc L^{2,1}} + C\, \varepsilon^{2} \,  \sum_{k_1,k_2,k_3} \prod_{j=1}^{3} \langle k_j\rangle^2 \int_0^{t} \Big |(v_{k_1} \overline{v_{k_2}} v_{k_3})(s)\Big | \, ds \\
\leq & \norm{v(0)}_{\Fc L^{2,1}} + C\,\varepsilon^2 \,   \int_0^t\norm{v(s)}_{\Fc L^{2,1}}^3\, ds ,
\end{split}
\end{equation}
after using the fact that $\langle k\rangle^2 \lesssim \prod_{j=1}^{3} \langle k_j\rangle^2$ for $k=k_1-k_2+k_3$. The rest of the proof is similar to that of \Cref{thm:LWP1}.
\end{proof}

\begin{rk}\label{rk:GWP2}
Following \Cref{rk:GWP}, the solution in \Cref{thm:LWP2} is global. Note that $H^3(\T) \subset \Fc L^{2,1}$.
\end{rk}

As a consequence, we have the following perturbation result:

\begin{cor}\label{thm:bootstrap2} Suppose that $u(0),u_{\mathrm{app}}(0)\in \Fc L^{2,1}$ and that $t\leq T_{\varepsilon}$ as in \eqref{eq:LWP_time2}. Then there exists some positive constant $C$ (independent of $\varepsilon$, $T_{\varepsilon}$ and the initial data) such that for all times $t\leq T_{\varepsilon}$ the following inequality holds:
\begin{align}\label{eq:bootstrap2}
\norm{u(t)-u_{\mathrm{app}}(t)}_{\Fc L^{2,1}} & \leq \norm{u(0)-u_{\mathrm{app}}(0)}_{\Fc L^{2,1}}
+ C\, \varepsilon^{4} \, t\, \left( \norm{u-u_{\mathrm{app}}}_{L^{\infty}([0,t),\Fc L^{2,1})}^5 + \norm{u_{\mathrm{app}}(0)}_{\Fc L^{2,1}}^5 \right)\nonumber\\
& + C\, \varepsilon^{2} \, t\, \left( \norm{u-u_{\mathrm{app}}}_{L^{\infty}([0,t),\Fc L^{2,1})}^2 + \norm{u_{\mathrm{app}}(0)}_{\Fc L^{2,1}}^2 \right)\, \norm{u-u_{\mathrm{app}}}_{L^{\infty}([0,t),\Fc L^{2,1})}\\
& + C\, \varepsilon^2 \,\left( \norm{u-u_{\mathrm{app}}}_{L^{\infty}([0,t),\Fc L^{2,1})}^3 + \norm{u_{\mathrm{app}}(0)}_{\Fc L^{2,1}}^3+\norm{u(0)}_{\Fc L^{2,1}}^3 \right).\nonumber
\end{align}
\end{cor}
\begin{proof}
We subtract \eqref{eq:resonant_Fourier} from \eqref{eq:pre_resonant_Fourier}, and obtain 
\begin{equation}
i\pa_t (w_k -a_k) = \varepsilon^2 \left( |a_k|^2 a_k - |w_k|^2 w_k \right) + \varepsilon^2 \, \sum_{k=k_1-k_2+k_3,\Omega\neq 0} w_{k_1} \overline{w_{k_2}} w_{k_3} e^{-i t\Omega}.
\end{equation}
Integrating this equation, we find that 
\begin{equation}
\begin{split}
i\,[w_k(t) -a_k(t)] & =i\, [w_k (0) -a_k(0)]+ \varepsilon^2 \int_0^{t} \left( |a_k|^2 a_k - |w_k|^2 w_k \right)(s)\, ds \\
& + \varepsilon^2 \, \sum_{k=k_1-k_2+k_3,\Omega\neq 0} \int_0^{t} (w_{k_1} \overline{w_{k_2}} w_{k_3})(s) e^{-is\Omega} \, ds.
\end{split}
\end{equation}
We multiply this equation by $\langle k\rangle^2$, take absolute values and sum in $k\in\Z$. Next we estimate
\[
\begin{split}
\sum_{k\in\Z} \langle k\rangle^2 \Big | \int_0^{t} ( |a_k|^2 a_k - & |w_k|^2 w_k )\, ds\Big |  \lesssim \sum_{k\in\Z} \langle k\rangle^2  \int_0^{t} \left( |a_k|^2 +  |w_k|^2 \right) |a_k - w_k|\, ds\\
& \lesssim \int_0^{t} \left( \sup_k |a_k|^2 +  \sup_k |w_k|^2 \right) \sum_{k\in\Z} \langle k\rangle^2 |a_k(s) - w_k(s)|\, ds\\
& \lesssim \int_0^{t} \left( \norm{a(s)}_{\Fc L^{2,1}}^2 +  \norm{w(s)}_{\Fc L^{2,1}}^2\right) \norm{a(s)-w(s)}_{\Fc L^{2,1}}\, ds\\
& \lesssim t\, \left( \norm{a}_{L^{\infty}([0,t),\Fc L^{2,1})}^2 +  \norm{w}_{L^{\infty}([0,t),\Fc L^{2,1})}^2 \right) \norm{a-w}_{L^{\infty}([0,t),\Fc L^{2,1})}\\
& \lesssim t\, \left( \norm{a}_{L^{\infty}([0,t),\Fc L^{2,1})}^2 +  \norm{w-a}_{L^{\infty}([0,t),\Fc L^{2,1})}^2 \right) \norm{w-a}_{L^{\infty}([0,t),\Fc L^{2,1})}.
\end{split}
\]
Finally, we use the fact that 
\begin{equation}\label{eq:conserved_a}
\norm{a(t)}_{\Fc L^{2,1}}=\norm{a(0)}_{\Fc L^{2,1}}=\norm{u_{\mathrm{app}}(0)}_{\Fc L^{2,1}}
\end{equation}
which follows by an explicit computation using \eqref{eq:def_a} and \eqref{eq:resonant_app}, and we undo the transformations that relate $a$ to $u_{\mathrm{app}}$ and $w$ to $u$ (which leave the $\Fc L^{2,1}$-norm invariant). This accounts for the third term in \eqref{eq:bootstrap2}.

In order to get the other terms, we first integrate by parts
\begin{equation}\label{eq:IBP}
\int_0^{t} w_{k_1} \overline{w_{k_2}} w_{k_3} e^{-i s\Omega} \, ds = -\frac{1}{i\Omega}\, w_{k_1} \overline{w_{k_2}} w_{k_3} e^{-i s\Omega} \Big |_{s=0}^{s=t} + \frac{1}{i\Omega}\,\int_0^{t} \pa_s \left( w_{k_1} \overline{w_{k_2}} w_{k_3}\right) e^{-i s\Omega}\, ds
\end{equation}
The boundary term can be controlled as follows:
\begin{multline*}
\varepsilon^2 \, \sum_{k\in\Z} \langle k \rangle^2 \Big | \sum_{k=k_1-k_2+k_3,\Omega\neq 0} \frac{1}{\Omega} \left[ (w_{k_1} \overline{w_{k_2}} w_{k_3})(t) e^{-i t\Omega} - (w_{k_1} \overline{w_{k_2}} w_{k_3})(0)\right] \Big | \\
\lesssim \varepsilon^2 \sum_{k_1, k_2, k_3} \prod_{j=1}^{3} \langle k_j\rangle^2 \, |w_{k_j}(t)| + \varepsilon^2 \sum_{k_1, k_2, k_3} \prod_{j=1}^{3} \langle k_j\rangle^2 \, |w_{k_j}(0)| \\
\lesssim \varepsilon^2 \left( \norm{w(t)-a(t)}_{\Fc L^{2,1}}^3 + \norm{a(0)}_{\Fc L^{2,1}}^3 + \norm{w(0)}_{\Fc L^{2,1}}^3\right).
\end{multline*}
This gives rise to the fourth term in \eqref{eq:bootstrap2}. Finally, let us consider the rightmost term in \eqref{eq:IBP}. There are three cases depending on which term $\pa_s$ hits; we do one as an example:
\begin{multline*}
\varepsilon^2 \sum_{k\in\Z} \langle k\rangle^2 \, \Big | \sum_{k=k_1-k_2+k_3,\Omega\neq 0} \frac{1}{\Omega}\int_0^{t} (\pa_s w_{k_1}) \overline{w_{k_2}} w_{k_3}\, e^{-i s\Omega}\, ds \Big | \\
 \lesssim \varepsilon^2 \,\int_0^t \sum_{k_1,k_2,k_3} \langle k_2\rangle^2 |w_{k_2}(s)| \, \langle k_3\rangle^2 |w_{k_3}(s)| \, \langle k_1\rangle^2 |\pa_s w_{k_1}(s)| \, ds\\
 \lesssim \varepsilon^2 \, t \norm{w}_{L^{\infty}([0,t),\Fc L^{2,1})}^2\, \norm{\pa_s w}_{L^{\infty}([0,t),\Fc L^{2,1})} .
\end{multline*}
Finally, one can easily use \eqref{eq:pre_resonant_Fourier} to show that 
\[ 
\norm{\pa_s w}_{L^{\infty}([0,t),\Fc L^{2,1})} \lesssim \varepsilon^2 \, \norm{w}_{L^{\infty}([0,t),\Fc L^{2,1})}^3.
\]
A final step using the triangle inequality to bound the norm of $w$ in terms of $w-a$ and $a$ (together with \eqref{eq:conserved_a}) yields the second term in \eqref{eq:bootstrap2}.
\end{proof}

As an application of the estimates above, we present a bootstrap argument that allows us to control the difference between the solution and our approximation.

\begin{prop}\label{thm:small_error2} Fix $\delta\in (0,1)$, $d_1,d_2>0$. Then there exists $\varepsilon_0$ (depending on $\delta$, $d_1$ and $d_2$) such that the following holds for all $\varepsilon\leq \varepsilon_0$. If $T= d_1 \varepsilon^{-1}$ and $\norm{u(0)}_{\Fc L^{2,1}}\leq d_2 \varepsilon^{-1/2}$, then
\begin{enumerate}[(i)]
\item we can extend the time where \eqref{eq:bootstrap2} is valid from $T_{\varepsilon}$ (in \eqref{eq:LWP_time2}) all the way to $T$; and
\item we have the estimate
\begin{equation}\label{eq:small_error2}
 \norm{u-u_{\mathrm{app}}}_{L^{\infty}([0,T],\Fc L^{2,1})} < \varepsilon^{-1/2+\delta} .
 \end{equation}
\end{enumerate}
\end{prop}
\begin{proof}
{\bf Step 1.} First of all, note that \eqref{eq:bootstrap2} is valid for any time $t$ such that $u(t)\in\Fc L^{2,1}$. The choice of $T_{\varepsilon}$ guarantees that this condition is met, but it is not necessary as we will soon show. We start by setting $u(0)=u_{\mathrm{app}}(0)$ as in \eqref{eq:resonant_Fourier}. Next define
\[ 
\tau_1 := \sup \{t\in [0,T]\mid \norm{u(t)-u_{\mathrm{app}}(t)}_{\Fc L^{2,1}}<\varepsilon^{-1/2+\delta_1}\},
\]
for some small $\delta_1\in (0,1)$. Note that if $t\leq \tau_1$ then 
\[ 
\norm{u(t)}_{\Fc L^{2,1}}  \leq \norm{u(t)-u_{\mathrm{app}}(t)}_{\Fc L^{2,1}} + \norm{u_{\mathrm{app}}(t)}_{\Fc L^{2,1}}< \varepsilon^{-1/2+\delta_1} + d_2 \varepsilon^{-1/2},
\]
and therefore $u(t)\in\Fc L^{2,1}$, which implies that \eqref{eq:bootstrap2} is valid for such times. Moreover, $u(t)$ must continue to live in $\Fc L^{2,1}$ thanks to the local well-posedness theory of \Cref{thm:LWP2}. 

If $t\leq \min \{ \tau_1, \lambda \varepsilon^{-1}\}$ for some small $\lambda>0$ to be fixed later, then \eqref{eq:bootstrap2} yields
\[ 
\begin{split}
\norm{u(t)-u_{\mathrm{app}}(t)}_{\Fc L^{2,1}}  \leq &\, C\, \lambda \, \varepsilon^{3} \, \left( \varepsilon^{-5/2+5\delta_1}  + d_2^5 \,\varepsilon^{-5/2} \right) + C\,\lambda\, \varepsilon \, \left( \varepsilon^{-1+\delta_1} + d_2^2 \varepsilon^{-1} \right)\, \varepsilon^{-1/2+\delta_1}\\ 
& + C\, \varepsilon^2 \,\left( \varepsilon^{-3/2+3\delta_1} + 2d_2^3 \varepsilon^{-3/2}\right)\\
 \leq & \, C\, d_2^2\, \lambda\, \varepsilon^{-1/2+\delta_1} + C\, \lambda \, \left( \varepsilon^{-1/2+3\delta_1}+\varepsilon^{1/2+5\delta_1}  + d_2^5 \,\varepsilon^{1/2} \right)\\
& + C\, \left( \varepsilon^{1/2+3\delta_1} + 2d_2^3 \varepsilon^{1/2}\right).
\end{split}
\]
If $\varepsilon$ is small enough, it is easy to guarantee that 
\[ C\, \lambda \, \left( \varepsilon^{-1/2+3\delta_1}+\varepsilon^{1/2+5\delta_1}  + d_2^5 \,\varepsilon^{1/2} \right)+ C\, \left( \varepsilon^{1/2+3\delta_1} + 2d_2^3 \varepsilon^{1/2}\right)\leq \frac{1}{4}\varepsilon^{-1/2+\delta_1}.\]
Note that for the term in $\varepsilon^{1/2}$ to be controlled by a multiple of $\varepsilon^{-1/2+\delta_1}$ we need to impose that $\delta_1<1$.
Making $4\lambda< (C\, d_2^2)^{-1}$ (note that it does not depend on $\varepsilon$), one can make sure that 
\[ C\, d_2^2\, \lambda\, \varepsilon^{-1/2+\delta_1}\leq \frac{1}{4} \varepsilon^{-1/2+\delta_1}.\]
Therefore we have that if $t\leq \min \{ \tau_1, \lambda \varepsilon^{-1}\}$, then
\[
\norm{u(t)-u_{\mathrm{app}}(t)}_{\Fc L^{2,1}}  \leq  \frac{1}{2}\,\varepsilon^{-1/2+\delta_1}.
\]
This finishes the bootstrap argument and shows that $\tau_1\geq \lambda \varepsilon^{-1}$. 

Therefore we know that 
\begin{equation}\label{eq:step1_bootstrap}
\norm{u(t)-u_{\mathrm{app}}(t)}_{\Fc L^{2,1}} < \varepsilon^{-1/2+\delta_1}
\end{equation}
must be true for all times $0\leq t \leq \lambda \varepsilon^{-1}$.

{\bf Step 2.} Let us fix $t_1:= \lambda\, \varepsilon^{-1}$, and restart our equation for times $t>t_1$. We want to show that we can take another step of size $\lambda\, \varepsilon^{-1}$ and therefore reach times $t_2=2\lambda\, \varepsilon^{-1}$ while the error is still not too large.

To do that, let 
\[ \tau_2 := \sup \{t\in [t_1,T]\mid \norm{u(t)-u_{\mathrm{app}}(t)}_{\Fc L^{2,1}}<\varepsilon^{-1/2+\delta_2}\}\]
for some $\delta_2\in (0,\delta_1)$. First of all, note that we have
\begin{equation}
\begin{split}
\norm{u(t_1)}_{\Fc L^{2,1}} & \leq \norm{u(t_1)-u_{\mathrm{app}}(t_1)}_{\Fc L^{2,1}} + \norm{u_{\mathrm{app}}(t_1)}_{\Fc L^{2,1}} < \varepsilon^{-1/2+\delta_1} + \norm{u_{\mathrm{app}}(0)}_{\Fc L^{2,1}}\\
& = \varepsilon^{-1/2+\delta_1}+ d_2\varepsilon^{-1/2} = (\varepsilon^{\delta_1} + d_2) \, \varepsilon^{-1/2}.
\end{split}
\end{equation}
Therefore, $u(t)$ is guaranteed to exist for some times past $t_1$ by \Cref{thm:LWP2}. Then consider $t_1\leq t\leq \min\{ \tau_2, t_1+\lambda_2 \varepsilon^{-1}\}$ for $\lambda_2>0$ small enough (to be fixed soon) and let's use a version of \eqref{eq:bootstrap2} for initial data at $t_1$ instead of zero:
\[
\begin{split}
|| u(t)- &\, u_{\mathrm{app}}(t)||_{\Fc L^{2,1}} \leq \norm{u(t_1)-u_{\mathrm{app}}(t_1)}_{\Fc L^{2,1}}
\\
& + C\, \varepsilon^{4} \, (t-t_1)\, \left( \norm{u-u_{\mathrm{app}}}_{L^{\infty}([t_1,t),\Fc L^{2,1})}^5 + \norm{u_{\mathrm{app}}(t_1)}_{\Fc L^{2,1}}^5 \right)\\
& + C\, \varepsilon^{2} \, (t-t_1)\, \left( \norm{u-u_{\mathrm{app}}}_{L^{\infty}([t_1,t),\Fc L^{2,1})}^2 + \norm{u_{\mathrm{app}}(t_1)}_{\Fc L^{2,1}}^2 \right)\, \norm{u-u_{\mathrm{app}}}_{L^{\infty}([t_1,t),\Fc L^{2,1})}\\ 
& + C\, \varepsilon^2 \,\left( \norm{u-u_{\mathrm{app}}}_{L^{\infty}([t_1,t),\Fc L^{2,1})}^3 + \norm{u_{\mathrm{app}}(t_1)}_{\Fc L^{2,1}}^3+\norm{u(t_1)}_{\Fc L^{2,1}}^3 \right).
\end{split}
\]
Since $\delta_2<\delta_1$, we can make $\varepsilon$ small enough so that 
\begin{equation}\label{eq:cond_rk}
\varepsilon^{-1/2+\delta_1} \leq \frac{1}{4}\, \varepsilon^{-1/2+\delta_2}.
\end{equation}
For $t_1\leq t\leq \min\{ \tau_2, t_1+\lambda_2 \varepsilon^{-1}\}$ we have that
\[
\begin{split}
\norm{u(t)-u_{\mathrm{app}}(t)}_{\Fc L^{2,1}}  \leq &\, \frac{1}{4}\, \varepsilon^{-1/2+\delta_2}
+ C\, \varepsilon^{3} \, \lambda_2\, \left(\varepsilon^{-5/2+5\delta_2} + d_2^5 \varepsilon^{-5/2} \right) \\
& + C\, \varepsilon \,\lambda_2\, \left( \varepsilon^{-1+2\delta_2} + d_2^2 \varepsilon^{-1} \right)\,\varepsilon^{-1/2+\delta_2}\\ 
& + C\, \varepsilon^2 \,\left(\varepsilon^{-3/2+3\delta_2} + d_2^3 \varepsilon^{-3/2} + (\varepsilon^{\delta_1} + d_2)^3 \, \varepsilon^{-3/2} \right)\\
 \leq & \, \frac{1}{4}\, \varepsilon^{-1/2+\delta_2} + C \,\lambda_2\, d_2^2\, \varepsilon^{-1/2+\delta_2}
+ C \, \lambda_2\, \left(\varepsilon^{1/2+5\delta_2} + d_2^5 \varepsilon^{1/2} +\varepsilon^{-1/2+3\delta_2} \right) \\ 
& + C\,\left( \varepsilon^{1/2+3\delta_2} + d_2^3 \varepsilon^{1/2}+ (\varepsilon^{\delta_1} + d_2)^3 \, \varepsilon^{1/2} \right).
\end{split}
\]
If $\varepsilon$ is small enough, we can guarantee that
\[ C \, \lambda_2\, \left(\varepsilon^{1/2+5\delta_2} + d_2^5 \varepsilon^{1/2} +\varepsilon^{-1/2+3\delta_2} \right) + C\,\left( \varepsilon^{1/2+3\delta_2} + d_2^3 \varepsilon^{1/2}+ (\varepsilon^{\delta_1} + d_2)^3 \, \varepsilon^{1/2} \right) < \frac{1}{4}\, \varepsilon^{-1/2+\delta_2}.\]
If $4\lambda_2 < (C\,d_2^2)^{-1}$ then,
\[ C \,\lambda_2\, d_2^2\, \varepsilon^{-1/2+\delta_2}< \frac{1}{4} \, \varepsilon^{-1/2+\delta_2}.\]
Note that we can choose $\lambda_2=\lambda$. With this choice, we have that for $t_1\leq t \leq \min \{ \tau_2 , t_1+\lambda \varepsilon^{-1}\}$ we must have that 
\[
 \norm{u(t)-u_{\mathrm{app}}(t)}_{\Fc L^{2,1}} < \frac{3}{4}\, \varepsilon^{-1/2+\delta_2}.
\]
This bootstrap argument shows that $\tau_2\geq t_1 +\lambda \varepsilon^{-1}=2\lambda \varepsilon^{-1}$. In particular, we have that
\begin{equation}\label{eq:step2_bootstrap}
 \norm{u(t)-u_{\mathrm{app}}(t)}_{\Fc L^{2,1}} < \varepsilon^{-1/2+\delta_2}
 \end{equation}
must hold for all $0\leq t\leq 2\lambda\varepsilon^{-2}$.

{\bf Step 3.} We can repeat this procedure by choosing a decreasing sequence of positive small numbers $\delta_n$ (and possibly making $\varepsilon$ smaller at each step) to guarantee that 
\begin{equation}\label{eq:step3_bootstrap}
 \norm{u(t)-u_{\mathrm{app}}(t)}_{\Fc L^{2,1}} < \varepsilon^{-1/2+\delta_n}
 \end{equation}
 for all times $0\leq t \leq n \lambda \varepsilon^{-1}$, where $4\lambda < (C\,d_2^2)^{-1}$ is independent of $\varepsilon$. We can stop this procedure after a finite number of steps once $n\lambda \varepsilon^{-1}\geq d_1 \varepsilon^{-1}=T$. This concludes the proof of the proposition.
 
 Finally note that $\delta_1$ was free for us to choose as long as $\delta_1\in (0,1)$, and then one could take any decreasing sequence (no matter how slow). This shows that the only condition on $\delta$ in \eqref{eq:small_error2} is for it to live in $(0,1)$.
\end{proof}

\begin{rk}
Note that Step 3 in the proof above could be improved by talking a number of steps $n$ that depends on $\varepsilon$. This would reduce the range of $\delta$ in \Cref{thm:small_error2} and allow us to reach times $T=\O(\varepsilon^{-1} \, |\log \varepsilon|)$. This small gain is due to condition \eqref{eq:cond_rk}, which limits the choice of $n(\varepsilon)$ as follows:
\[
1>\delta_1-\delta_n=\sum_{i=1}^{n-1} (\delta_i-\delta_{i+1}) \geq \log 4\cdot \frac{n}{|\log \varepsilon|}\ .
\]
In this paper we will focus on times of the form $\varepsilon^{-r}$, and so we will not pursue this further.
\end{rk}

\subsection{Proof of \Cref{thm:mainLDP} for critical times}

 Fix $z_0>0$ (independent of $\varepsilon$) and let $t=d_1 \varepsilon^{-1}$ for some fixed $d_1>0$. We want to study the limit
\begin{equation}\label{eq:criticalLDP}
\lim_{\varepsilon\rightarrow 0^{+}} \varepsilon \, \log \P \left( \sup_{x\in\T} |u(t,x)| > z_0\, \varepsilon^{-1/2} \right).
\end{equation}
Note that the solution $u$ exists at time $t$ and $\norm{u(t)}_{L^{\infty}_x}<\infty$ almost surely thanks to \Cref{rk:GWP2}. Let $\mathcal{D}_{\varepsilon}$ be the event $\{\sup_{x\in\T} |u(t,x)| > z_0\, \varepsilon^{-1/2}\}$. Let us prove an upper bound for $\P(\mathcal{D}_{\varepsilon})$ first.

Fix $\delta\in (0,1)$. By the triangle inequality and the embedding $\Fc L^{2,1}\subset L^{\infty}$ we have that $\mathcal{D}_{\varepsilon}\subset \mathcal{D}_{\varepsilon}^+$ where 
\[
\mathcal{D}^+_{\varepsilon}:=\left\lbrace \norm{u_{\mathrm{app}}(t)}_{L^{\infty}_x} + \norm{u(t)-u_{\mathrm{app}}(t)}_{\Fc L^{2,1}} > z_0\,\varepsilon^{-1/2} \right\rbrace .
\]
Let us also define the event
\[
\mathcal{B}_{\varepsilon} := \left\lbrace \norm{u(t)-u_{\mathrm{app}}(t)}_{\Fc L^{2,1}}\leq z_0\, \varepsilon^{-1/2+\delta}\right\rbrace.
\]
Then we have that 
\begin{equation}\label{eq:critical_upperb}
\P (\mathcal{D}_{\varepsilon}) \leq \P (\mathcal{D}^+_{\varepsilon} \cap \mathcal{B}_{\varepsilon}) + \P (\mathcal{D}^+_{\varepsilon} \cap \mathcal{B}_{\varepsilon}^c).
\end{equation}

Let us study the first term in \eqref{eq:critical_upperb}. We have that
\[
\P (\mathcal{D}^+_{\varepsilon} \cap \mathcal{B}_{\varepsilon}) \leq \P \left( \norm{u_{\mathrm{app}}(t)}_{L^{\infty}_x} > z_0 (\varepsilon^{-1/2}-\varepsilon^{-1/2+\delta})\right).
\]
By \Cref{thm:resonantLDP}, 
\begin{equation}\label{eq:critical_upperb2}
\lim_{\varepsilon\rightarrow 0^{+}} \varepsilon \, \log \P \left( \norm{u_{\mathrm{app}}(t)}_{L^{\infty}_x} > z_0 (\varepsilon^{-1/2}-\varepsilon^{-1/2+\delta})\right) = -\frac{z_0^2}{\sum_{k\in\Z}c_k^2} .
\end{equation}
Next we consider the second term in \eqref{eq:critical_upperb}. First note that $\langle k\rangle^2 c_k \leq C_b \, \sqrt{c_k}$ for some positive constant $C_b>0$ which only depends on $b$. Fix $d_2>0$ large enough so that:
\begin{equation}\label{eq:cond_d2}
d_2^2 > 2\, z_0^2 \, C_b^2\, \frac{\sum_{k\in\Z} c_k}{\sum_{k\in\Z} c_k^2}.
\end{equation}
We will justify this choice shortly.

By \Cref{thm:small_error2}, if $\varepsilon<\varepsilon_0=\varepsilon_0 (\delta,d_1,d_2)$ is small enough we can arrange:
\[
\P (\mathcal{D}^+_{\varepsilon} \cap \mathcal{B}_{\varepsilon}^c) \leq \P ( \mathcal{B}_{\varepsilon}^c) \leq \P \left( \norm{u_0}_{\Fc L^{2,1}} \geq  d_2\,\varepsilon^{-1/2}\right).
\]
We claim that
\begin{equation}\label{eq:critical_upperb3}
\limsup_{\varepsilon\rightarrow 0^{+}}\varepsilon\, \log  \P \left( \norm{u_0}_{\Fc L^{2,1}} \geq  d_2\,\varepsilon^{-1/2}\right) \leq  -\frac{d_2^2}{C_b^2\, \sum_{k\in\Z} c_k}.
\end{equation}
Given that $\langle k \rangle^2 c_k\leq C_b \sqrt{c_k}$, the random variable  $\norm{u_0}_{\Fc L^{2,1}}\leq C_b \,\norm{\tilde{u}_0}_{\Fc L^{0,1}}$ almost surely, where $\tilde{u}_0$ has coefficients $\sqrt{c_k}$. As a result we can apply \Cref{thm:GE} and \eqref{eq:GE_LDP_upper} again to $\tilde{u}_0$ (with $b$ halved) in order to prove \eqref{eq:critical_upperb3}.

Thanks to our choice of $d_2$ in \eqref{eq:cond_d2}, \eqref{eq:critical_upperb3} implies have that the first term in \eqref{eq:critical_upperb} is asymptotically larger than the second one. In particular, we can arrange
\[
\P (\mathcal{D}_{\varepsilon}) \leq 2\, \P \left( \norm{u_{\mathrm{app}}(t)}_{L^{\infty}_x} > z_0 (\varepsilon^{-1/2}-\varepsilon^{-1/2+\delta})\right)
\]
for $\varepsilon$ small enough, which concludes the proof of the upper bound in view of \eqref{eq:critical_upperb2}.

Finally, let us prove a lower bound. First of all note that $\mathcal{D}^-_{\varepsilon}\subset \mathcal{D}_{\varepsilon}$ if we define
\[
\mathcal{D}^-_{\varepsilon}:=\left\lbrace  \norm{u_{\mathrm{app}}(t)}_{L^{\infty}_x}-\norm{u(t)-u_{\mathrm{app}}(t)}_{\Fc L^{2,1}} > z_0\,\varepsilon^{-1/2} \right\rbrace .
\]
Then we have that
\begin{equation}\label{eq:critical_lowerb}
\begin{split}
\P (\mathcal{D}_{\varepsilon}) & \geq \P (\mathcal{D}^-_{\varepsilon} \cap \mathcal{B}_{\varepsilon})\geq \P \left( \left\lbrace \norm{u_{\mathrm{app}}(t)}_{L^{\infty}_x}> z_0 (\varepsilon^{-1/2} + \varepsilon^{-1/2+\delta})\right\rbrace \cap \mathcal{B}_{\varepsilon}\right) \\
& \geq \P \left( \norm{u_{\mathrm{app}}(t)}_{L^{\infty}_x}> z_0 (\varepsilon^{-1/2} + \varepsilon^{-1/2+\delta})\right) - \P ( \mathcal{B}_{\varepsilon}^c)
\end{split}
\end{equation}

By \Cref{thm:resonantLDP}, we have that 
\begin{equation}\label{eq:critical_lowerb2}
\lim_{\varepsilon\rightarrow 0^{+}} \varepsilon \, \log \P \left( \norm{u_{\mathrm{app}}(t)}_{L^{\infty}_x} > z_0 (\varepsilon^{-1/2}+\varepsilon^{-1/2+\delta})\right) = -\frac{z_0^2}{\sum_{k\in\Z}c_k^2} .
\end{equation}
This describes the asymptotic behavior of the first summand in \eqref{eq:critical_lowerb}.
By \Cref{thm:small_error2}, if $\varepsilon<\varepsilon_0$ the second summand satisfies
\[
\P ( \mathcal{B}_{\varepsilon}^c)\leq \P \left( \norm{u_0}_{\Fc L^{2,1}} \geq  d_2 \varepsilon^{-1/2}\right).
\]
Using \eqref{eq:cond_d2} and arguing as in the proof of the upper bound, we conclude that this second term is asymptotically smaller than the first term in \eqref{eq:critical_lowerb}. In particular, we can arrange
\[
\P (\mathcal{D}_{\varepsilon})\geq \frac{1}{2} \P \left( \norm{u_{\mathrm{app}}(t)}_{L^{\infty}_x}> z_0 (\varepsilon^{-1/2} + \varepsilon^{-1/2+\delta})\right)
\]
for $\varepsilon$ small enough. Finally, \eqref{eq:critical_lowerb2} yields the desired asymptotic behavior and finishes the proof.

\section{Most Likely Initial Data for a Rogue Wave}\label{sec:MLE}

In the previous sections, we established the existence of a Large Deviations Principle for the solution of the NLS equation. This allowed us to quantify the probability that we observe a rogue wave: yet rare, rogue waves are actually plausible.

In this section, we change gears slightly and we ask ourselves the following question: ``Conditioned on the fact that we observe a rogue wave, what is the most likely initial datum that generates such phenomenon?'' For the sake of concreteness, we will focus on the solution of the linear Schr\"odinger equation. (In \Cref{rk:fromLineartoNLS}, we adapt our arguments for the linear equation to the nonlinear equation \eqref{eq:NLS_mainLDP} for critical and subcritical times.)

In order to answer this question, it is more convenient to take a different point of view regarding the initial data.  Let us consider the set
\[ 
\mathcal{H}=\{(r_k,\phi_k)_{k\in\Z}\mid r_k\geq 0,\, \phi_k\in [0,2\pi)\}.
\]
It is easy to see that there is a correspondence between elements of $\mathcal{H}$ and initial data for the solution $u$ to the linear Schr\"odinger equation 
\begin{equation}\label{eq:linearSchrodinger}
    \begin{cases}
    i\partial_t u + \Delta u = 0, \quad x\in \T ,\\
    u(t,x \,|\, \vartheta )|_{t=0} = u_0(x \,|\, \vartheta ) = \sum_{k\in\Z} c_k r_k e^{i kx + i\phi_k} ,
    \end{cases}
\end{equation}
via
\begin{equation}\label{eq:initialDataMapping}
\vartheta:=(r_k,\phi_k)_{k\in\Z} \in \mathcal{H} \mapsto u_0(x \,|\, \vartheta ) = \sum_{k\in\Z} c_k r_k e^{i kx + i\phi_k}.
\end{equation}

That is, we interpret the initial datum as a parameter for the solution of the linear Schr\"odinger equation. The variable $\vartheta$ lives on the space\footnote{Notice that we use the notation $(r_k,\phi_k)$, as opposed to $(R_k,\varphi_k)$, to emphasize the fact that we are not seeing the initial condition as random variables, but as parameters.} of parameters $\mathcal{H}$. To recover the probability framework we presented in \eqref{eq:intro_data}, we can endow $\mathcal{H}$ with the probability measure
\begin{equation}\label{eq:measureSpaceParameters}
m := \bigtimes_{k \in \Z} m_{R_k} \times m_{\varphi_k},
\end{equation}
where $R_k \sim Rayleigh(1/\sqrt{2})$, $\varphi_k \sim U[0,2\pi]$, $m_X$ is the probability measure of the random variable $X$, and $\bigtimes$ denotes the countable product of probability measures, as in Section 1.5 from \cite{daprato2006introduction}. The set of initial conditions that produce a rogue wave of height $z(\varepsilon):= z_0 \varepsilon^{-1/2}$ at time $t >0$, with $z_0 >0$ and $\varepsilon \rightarrow 0^+$, corresponds to 
\begin{equation}\label{eq:def_D}
\mathcal{D}(t, z(\varepsilon)):= \left \lbrace\vartheta \in \mathcal{H} : \sup_{x\in\T}\, \lvert u\left(t, x \,|\, \vartheta \right) \rvert \geq z_0 \varepsilon^{-1/2} \right \rbrace,
\end{equation}
where $u$ is the solution to the linear Schr\"odinger equation. By the definition of $m$ in \eqref{eq:measureSpaceParameters}, it is clear that $m(\{\vartheta\}) = 0$, for all $\vartheta \in \mathcal{H}$. Hence, the goal of the rest of the section is to give meaning to the idea of ``most likely initial datum'' that generates a rogue wave. 

\subsection{Heuristic Approximation Using MLE Techniques}

One possibility would be to say that $\vartheta^* \in \mathcal{D}(t, z(\varepsilon))$ is the most likely datum if it maximizes the density of the measure $m$, as it is customary in Statistics when studying Maximum Likelihood Estimators (MLEs). The problem with this approach is that it requires $m$ to be absolutely continuous with respect to a reference measure\footnote{Usually, the Lebesgue measure in $\R^K$ where $K$ is the dimension of the space of parameters}. This idea cannot be applied directly to our problem, since $\mathcal{H}$ is an infinite-dimensional vector space.

We propose a minimization problem, $\mathcal{P}$, to characterize the most likely initial data for a rogue wave. This pseudo-MLE problem comes motivated by studying the actual MLE problem when we truncate the number of non-zero Fourier modes for the initial datum, and then we make this number go to infinity. More concretely, let us define\footnote{In the definition of $\mathcal{D}^{(N)}(t, z(\varepsilon))$, we are implicitly using that there exists embeddings, $\mathcal{H}^{(N)} \hookrightarrow \mathcal{H}$ and $\mathcal{D}^{(N)}(t, z(\varepsilon)) \hookrightarrow \mathcal{D}(t, z(\varepsilon))$, via setting the missing Fourier modes equal to zero, so that we can define the solution of the linear Schr\"odinger equation through \eqref{eq:initialDataMapping}.}
\[ 
\mathcal{H}^{(N)}=\{(r_k,\phi_k)_{|k| \le N}\mid r_k\geq 0,\, \phi_k\in [0,2\pi) \},
\]
\[
m^{(N)}:= \bigtimes_{|k| \le N} m_{R_k} \times m_{\varphi_k},
\]
and
\[
\mathcal{D}^{(N)}(t, z(\varepsilon)):= \left \lbrace\vartheta \in \mathcal{H}^{(N)} : \sup_{x\in\T}\, \lvert u\left(t, x \,|\, \vartheta \right) \rvert \geq z_0 \varepsilon^{-1/2} \right \rbrace.
\]

 For this truncated scenario, the measure $m^{(N)}$ has a density with respect to the Lebesgue measure in $(\R^{2})^{2N+1}$. It is easy to see that the problem of maximizing this density on the set of initial data that generates a rogue wave is equivalent\footnote{Since we are using the polar representation of $\eta_k$ as $R_k e^{i \varphi_k}$, we need to express the Lebesgue measure in polar coordinates to obtain the claimed equivalence between both optimization problems.} to the following minimization problem:
\begin{equation}\label{eq:truncatedMinProblem}
    \mathcal{P}^{(N)}:=\begin{cases}
    \text{min } \sum_{|k|\le N} r_k^2,\\
   \quad \text{s.t. }(r_k,\phi_k)_{|k|\le N} \in \mathcal{D}^{(N)}(t, z(\varepsilon)).
    \end{cases}
\end{equation}

\noindent Formally taking $N \rightarrow \infty$, we can define $\mathcal{P}$ as
\begin{equation}\label{eq:mainMinProblem}
    \mathcal{P}:=\begin{cases}
    \text{min } \sum_{k \in \Z} r_k^2,\\
   \quad \text{s.t. }(r_k,\phi_k)_{k \in \Z} \in \mathcal{D}(t, z(\varepsilon)).
    \end{cases}
\end{equation}

We can show the existence of a minimizer for \eqref{eq:mainMinProblem} by solving a relaxation of the original problem. Let us define the set 
\[
\widetilde{\mathcal{D}}(t,z(\varepsilon))= \lbrace \vartheta \in \mathcal{H} : \sum_{k\in\Z} c_k r_k \geq z_0 \varepsilon^{-1/2} \rbrace
\]
and the relaxed problem 
\begin{equation}\label{eq:relaxationMinProblem}
    \widetilde{\mathcal{P}}:=\begin{cases}
    \text{min } \sum_{k \in \Z} r_k^2,\\
   \quad \text{s.t. }(r_k,\phi_k)_{k \in \Z} \in \widetilde{\mathcal{D}}(t, z(\varepsilon))
    \end{cases}.
\end{equation}
Note that $\mathcal{D}(t,z(\varepsilon))\subset \widetilde{\mathcal{D}}(t,z(\varepsilon))$ by \eqref{eq:overest}, and hence $\widetilde{\mathcal{P}}$ is, in fact, a relaxation of $\mathcal{P}$. Applying Lagrange multipliers formally\footnote{One could rigorously justify this step by assuming that $(r_k)_{k \in \Z} \in l^2(\Z)$, and then using the theory of Lagrange multipliers on Banach spaces. Note that this shrinkage of the feasible region does not affect the set of solutions of $\widetilde{\mathcal{P}}$. Hence, without loss of generality, in what follows we will use this assumption when solving $\mathcal{P}$ and $\widetilde{\mathcal{P}}$.}, we find the following ansatz: the set of all minimizers for \eqref{eq:relaxationMinProblem} is given by \begin{equation}\label{eq:relaxationMinSolution}
\widetilde{\mathcal{S}}:= \left\{(r_k^*,\phi_k)_{k \in \Z} \in \widetilde{\mathcal{D}}(t, z(\varepsilon)) : r_k^* = \frac{c_k z(\varepsilon)}{\sum_{k\in\Z} c_k^2} \right\}.
\end{equation}

To show that $\widetilde{\mathcal{S}}$ contains all the solutions for $\widetilde{\mathcal{P}}$, it is easy to see that, for any $(r_k,\phi_k)_{k \in \Z} \in \widetilde{\mathcal{D}}(t, z(\varepsilon))$, we have
\begin{align*}
& \sum_{k \in \Z} r_k^2  = \sum_{k \in \Z} (r_k^*)^2 + 2 r_k^*(r_k - r_k^*) + (r_k - r_k^*)^2 \ge \sum_{k \in \Z} (r_k^*)^2 + 2 \frac{z(\varepsilon)}{\sum_{k\in\Z} c_k^2}\sum_{k \in \Z}c_k r_k - c_k r_k^* \qquad \qquad \\
& \qquad \qquad \qquad \qquad \ge \sum_{k \in \Z} (r_k^*)^2 + 2 \left(\frac{z(\varepsilon)}{\sum_{k\in\Z} c_k^2}\right)^2 \sum_{k\in\Z} c_k^2 - 2\frac{z(\varepsilon)}{\sum_{k\in\Z} c_k^2}\sum_{k \in \Z}c_k r_k^* = \sum_{k \in \Z} (r_k^*)^2, 
\end{align*}
where the first inequality is strict unless $(r_k,\phi_k)_{k \in \Z} \in \widetilde{\mathcal{S}}$. Note that there are an infinite number of minimizers in $\widetilde{\mathcal{S}}$, since the phases are free. 

Back to the original problem, and using \eqref{eq:overest} once more, one can find a minimizer for $\mathcal{P}$ that lives in $\widetilde{\mathcal{S}}$ choosing the phases $(\phi^*_k)_{k \in \Z}$ such that
\begin{equation}\label{eq:phaseConstraint}
\sup_{x\in\T}\ \left\lvert \sum_{k\in\Z} c_k^2 \, e^{ikx -ik^2 t+ i\phi^*_k} \right\rvert \geq \sum_{k\in\Z} c_k^2.
\end{equation}
This is only possible if there exists some $x^*\in \T$ such that $kx^* -k^2 t +\phi^*_k=\phi^*_0$, for all $k \in \Z$. It is now clear that we can find a \textbf{two-parameter family of minimizers} for $\mathcal{P}$ depending on the position of the peak of the rogue wave, given by $x^*\in\T$, and the phase of the zeroth mode, given by $\phi^*_0\in [0,2\pi)$. These corresponds to the space-translation and phase rotation symmetries of the Schr\"odinger equation, both of which leave norm $\norm{u(t)}_{L^{\infty}_x}$ invariant. 

\noindent Now, we are finally ready to prove the following result.

\begin{prop}\label{prop:existence_minimizer}
Fix $z_0 >0$, $\varepsilon >0$, $x^*\in \T$, $\phi^*_0\in [0,2\pi)$, and $t>0$. Then there exists a unique solution to the minimization problem $\mathcal{P}$ in \eqref{eq:mainMinProblem} with peak at position $x^*\in\T$, time $t>0$, and zeroth phase $\phi^*_0$, given by $(r^*_k, \phi^*_k)_{k \in \Z}$, where
\begin{equation}\label{eq:optimal_r_k}
r_k^* := \frac{c_k z(\varepsilon)}{\sum_{j\in\Z} c_j^2},
\end{equation}
and
\begin{equation}\label{eq:MLE_phases}
\phi^*_k := \phi^*_0 - kx^* + k^2 t \, \left(\text{mod } 2\pi\right).
\end{equation}
Moreover, this minimizer satisfies
\begin{equation}\label{eq:conjecture}
\min_{(r_k,\phi_k)_{k\in\Z}\in \mathcal{D}(t,z(\varepsilon))} \sum_{k\in\Z} r_k^2 = \frac{z(\varepsilon)^2}{\sum_{k\in\Z} c_k^2},
\end{equation}
and corresponds to the initial datum
\begin{equation}\label{eq:u0_minimizer}
u_0^{\ast}(x)= \sum_{k\in\Z} \frac{c_k^2 z(\varepsilon)}{\sum_{j\in\Z} c_j^2} \, e^{ik(x-x^*)+ ik^2t+i\phi^*_0}.
\end{equation}
\end{prop}
\begin{proof}
The fact that $(r^*_k, \phi^*_k)_{k \in \Z}$ is a minimizer for $\mathcal{P}$ follows from the fact that it is a minimizer for the relaxed problem $\widetilde{\mathcal{P}}$. Hence, existence is established. It is also easy to see that \eqref{eq:conjecture} and \eqref{eq:u0_minimizer} follow from direct computation.

To establish uniqueness, let us assume that $\vartheta^{new}:=(r^{new}_k, \phi^{new}_k)_{k \in \Z}$ is another minimizer for $\mathcal{P}$. By \eqref{eq:relaxationMinProblem} and \eqref{eq:relaxationMinSolution}, we have that $\vartheta^{new}$ must also live in $\widetilde{\mathcal{S}}$, and hence $r^{new}_k = r^*_k$, for all $k \in \Z$. To see that the phases must coincide as well, one just needs to realize that the constraint in \eqref{eq:phaseConstraint} is actually binding for $(r^*_k, \phi^*_k)_{k \in \Z}$. Hence, it must also be binding for $\vartheta^{new}$ (otherwise, we could reduce the value of the objective function), which implies that $\phi^{new}_k = \phi^*_k$, for all $k \in \Z$, as we wanted to prove.
\end{proof}

\begin{rk}
One can show that, under the same conditions of \Cref{prop:existence_minimizer}, the moduli of the unique minimizer for $\mathcal{P}^{(N)}$ converges in $l^2$ to $(r^*_k)_{k \in \Z}$, and the phases end up aligning with $(\phi^*_k)_{k \in \Z}$. This justifies the step where we formally took $N \rightarrow \infty$ to propose $\mathcal{P}$.
\end{rk}


\subsection{Concentration result}

Our next goal is to show that the minimizers described in \Cref{prop:existence_minimizer} concentrate most of the probability of the set $\mathcal{D}(t,z(\varepsilon))$. To make sense of this notion, we construct a 
neighborhood $\mathcal{U}(\varepsilon,m_1, m_2)$ of the minimizers, given by the $(r_k, \phi_k)_{k\in\Z}\in\mathcal{H}$ satisfying two conditions:
\begin{equation}\label{eq:condU1}
   r^*_k \leq r_k \leq r^*_k + \varepsilon^{\beta} \qquad \mbox{for all}\ |k|\leq m_1,
\end{equation}
for $\beta>0$ to be defined, and
\begin{equation}\label{eq:condU2}
 |\phi_k -\phi_k^{\ast}|\leq \varepsilon\qquad \mbox{for all}\ |k|\leq m_2\ \mbox{except $k=0,1$},
\end{equation}
where $r^*_k$ and $\phi_k^{\ast}=\phi_k^{\ast}(\phi_0,\phi_1)$ were defined in \eqref{eq:optimal_r_k} and \eqref{eq:MLE_phases}, respectively. Condition \eqref{eq:condU1} is a one-sided interval around the minimizers' moduli, $(r^*_k)_{k\in\Z}$. Condition \eqref{eq:condU2} captures the fact that a large number of phases must align at the right time $t>0$ for a rogue wave to happen. Notice that $(r^*_k, \phi^*_k)_{k \in \Z} \in \mathcal{U}(\varepsilon,m_1, m_2)$, for any choice of $m_1$ and $m_2$.

The numbers $m_1$ and $m_2$ will be chosen later in such a way that they go to infinity as $\varepsilon\rightarrow 0^+$. In particular, we will see that $m_1 \sim |\log \varepsilon|$ and $m_2 \sim \varepsilon^{-\alpha}$ for $\alpha \in (0,1)$. Hence, $m_1 \ll m_2$ as $\varepsilon\rightarrow 0^+$. That is, for $|k| \in (m_1, m_2)$, the size of the moduli is irrelevant, as long as the phases are synchronized. This is due to the fast damping effect of $(c_k)_{k\in\Z}$. However, it is not enough to control the modulus of $2m_1+1$ Fourier modes to generate a rogue wave. A lack of synchronization in the phases for $|k| \in (m_1, m_2)$ could potentially destroy a rogue wave generated by the first $2m_1+1$ Fourier modes. This two-scale structure for the linear Schr\"odinger equation is a new phenomenon\footnote{This mechanism to generate rogue waves highly depends on the coefficients $(c_k)_{k\in\Z}$. See \Cref{rk:coeff} for a motivation of our choice.} in the study of the generation of rogue waves. 



Our first result shows that if one chooses $m_1$ and $m_2$ as proposed above, the set $\mathcal{U}(\varepsilon,m_1, m_2)$ has the same asymptotic probability as $\mathcal{D}(t,z(\varepsilon))$.

\begin{prop}\label{thm:LDP_U} 
Fix $z_0>0$. Suppose that 
\begin{equation}\label{eq:m1}
m_1=m_1(\varepsilon)=-\frac{1}{b} \log \left( \frac{2\sum_{j\in\Z} c_j^2}{z_0}\, \varepsilon^{1/2+ \beta}\right)
\end{equation}
for $\beta\in (0,1/2)$ and $m_2=m_2(\varepsilon)=\varepsilon^{-\alpha}$ for $\alpha\in (0,1)$. Then 
\begin{equation}\label{eq:LDP_U}
\lim_{\varepsilon\rightarrow 0^{+}} \varepsilon\, \log \P ( \mathcal{U}(\varepsilon,m_1, m_2)) = - \frac{z_0^2}{\sum_{k\in\Z} c_k^2}.
\end{equation}
\end{prop}


\begin{proof}
Using \eqref{eq:condU1}, \eqref{eq:condU2}, and the independence between moduli and phases, we may bound
\begin{align}\label{eq:lower_bound_U}
\P(\mathcal{U}(\varepsilon, m_1, m_2)) & = \pp{ \sup_{|k|\leq m_2, k\neq 0,1} |\varphi_k -\phi_k^{\ast}(\varphi_0,\varphi_1)|\leq \varepsilon} \nonumber \prod_{|k|\leq m_1} \pp{r^*_k \leq R_k \leq r^*_k + \varepsilon^{\beta}} \nonumber \\
& = (2\varepsilon)^{2m_2-1}\ \prod_{|k|\leq m_1} \left[e^{-\left(\frac{c_k }{\sum_{j\in\Z} c_j^2}\, z_0 \varepsilon^{-1/2}\right)^2} - e^{-\left(\frac{c_k }{\sum_{j\in\Z} c_j^2}\, z_0 \varepsilon^{-1/2} \,+ \, \varepsilon^{\beta}\right)^2}\right],
\end{align}
where we used the definition of $r^*_k$ in \eqref{eq:optimal_r_k}. Next, we take logarithms in \eqref{eq:lower_bound_U} and we multiply by $\varepsilon$, which yields
\begin{equation}\label{eq:prob_U}
\begin{split}
& \varepsilon \, \log \P(\mathcal{U}(\varepsilon, m_1, m_2)) = \varepsilon\, [2m_2(\varepsilon)-1]\,\log (2\varepsilon) \\ & \qquad - \varepsilon \sum_{|k|\leq m_1(\varepsilon)} \left(\frac{c_k }{\sum_{j\in\Z} c_j^2}\, z_0 \varepsilon^{-1/2} \right)^2 + \varepsilon \sum_{|k|\leq m_1(\varepsilon)} \log \left( 1- e^{-2\, \frac{c_k }{\sum_{j\in\Z} c_j^2}\, z_0 \varepsilon^{-1/2+\beta}-\varepsilon^{2\beta}}\right).
\end{split}
\end{equation}

Given that $m_2(\varepsilon)=\varepsilon^{-\alpha}$ for $\alpha\in (0,1)$, the first term on the right-hand side of \eqref{eq:prob_U} tends to zero. That is, $\lim_{\varepsilon\rightarrow 0^{+}} \varepsilon\, [2m_2(\varepsilon)-1]\,\log (2\varepsilon)  = 0$.
The second term on the right-hand side of \eqref{eq:prob_U} yields
\[
\lim_{\varepsilon\rightarrow 0^{+}} \varepsilon \sum_{|k|\leq m_1(\varepsilon)} \left(\frac{c_k }{\sum_{j\in\Z} c_j^2}\, z_0 \varepsilon^{-1/2}\right)^2  = \lim_{\varepsilon\rightarrow 0^{+}} \frac{\sum_{|k|\leq m_1(\varepsilon)} c_k^2}{(\sum_{k\in\Z} c_k^2)^2}\,z_0^2= \frac{z_0^2}{\sum_{k\in\Z} c_k^2}
\]
given that $m_1(\varepsilon)\rightarrow \infty$. All that remains to show is that the last term on the right-hand side of \eqref{eq:prob_U} tends to zero. That is, if we define
\[
h(\varepsilon)=\varepsilon \sum_{|k|\leq m_1(\varepsilon)} \log \left( 1- e^{-2\, \frac{c_k }{\sum_{j\in\Z} c_j^2}\, z_0 \varepsilon^{-1/2+\beta}-\varepsilon^{2\beta}}\right),
\]
we would like to show that $\lim_{\varepsilon\rightarrow 0^{+}} h(\varepsilon)=0$. One can compute this limit using the inequalities
\begin{equation}\label{eq:logy}
\frac{y}{1+y} \leq \log (1+y) \leq y, \quad \mbox{for}\ y>-1,
\end{equation}
with $y = -\exp\left(-2\,\frac{c_k }{\sum_{j\in\Z} c_j^2}\, z_0 \varepsilon^{-1/2+\beta}-\varepsilon^{2\beta}\right)$. The right-hand side of \eqref{eq:logy} gives us that $h(\varepsilon)\leq 0$. Using the left-hand side of \eqref{eq:logy}, it follows that 
\[
\begin{split}
h(\varepsilon) & \geq - \varepsilon\sum_{|k|\leq m_1(\varepsilon)} \frac{\exp\left(-2\,\frac{c_k }{\sum_{j\in\Z} c_j^2}\, z_0 \varepsilon^{-1/2+\beta}-\varepsilon^{2\beta}\right)}{1- \exp\left(-2\,\frac{c_k }{\sum_{j\in\Z} c_j^2}\, z_0 \varepsilon^{-1/2+\beta}-\varepsilon^{2\beta}\right)}\\
& \geq - \varepsilon \, m_1(\varepsilon)\,  \frac{\exp\left(-2\,\frac{c_{m_1} }{\sum_{j\in\Z} c_j^2}\, z_0 \varepsilon^{-1/2+\beta}-\varepsilon^{2\beta}\right)}{1- \exp\left(-2\,\frac{c_{m_1} }{\sum_{j\in\Z} c_j^2}\, z_0 \varepsilon^{-1/2+\beta}-\varepsilon^{2\beta}\right)}\\
& \gtrsim -\varepsilon \, m_1(\varepsilon)\, \frac{1}{1-\exp \left(-5\,\varepsilon^{2\beta}\right)} \longrightarrow 0
\end{split}
\]
as long as $\beta<1/2$. The last inequality is due to the choice of $m_1$ in \eqref{eq:m1}, which implies that 
\[
\frac{c_{m_1}}{\sum_{j\in\Z}c_j^2} z_0 \varepsilon^{-1/2}\geq 2 \varepsilon^{\beta}.
\]
This concludes the proof of \Cref{thm:LDP_U}.
\end{proof}



In view of \Cref{thm:LDP_U}, we have constructed a one-sided neighborhood of the minimizers, $\mathcal{U}(\varepsilon,m_1,m_2)$, which (asymptotically) concentrates as much probability as the set of rogue waves $\mathcal{D}(t,z(\varepsilon))$. Note, however, that we are allowing some flexibility in $\mathcal{U}$, since we only control a finite number of Fourier modes, and therefore $\mathcal{U}$ may contain waves that will not quite reach height $z_0\varepsilon^{-1/2}$. Hence, a better set to compare $\mathcal{U}$ with is
\begin{equation}
\mathcal{D} (t,z(\varepsilon) - a(\varepsilon))= \left\lbrace \vartheta\in\mathcal{H} : \sup_{x\in\T} |u(t,x\, |\,\vartheta)|\geq z_0 \varepsilon^{-1/2}  - a(\varepsilon)\right\rbrace,
\end{equation}
with $a(\varepsilon) \rightarrow 0$ as $\varepsilon \rightarrow 0^+$. Notice that $\mathcal{D} (t,z(\varepsilon) - a(\varepsilon))$ is virtually indistinguishable from $\mathcal{D}(t,z(\varepsilon))$ and enjoys the same LDP.

Our next goal is to show that $\mathcal{U}(\varepsilon,m_1,m_2)$ is almost fully contained in $\mathcal{D} (t,z(\varepsilon) - a(\varepsilon))$. In order to prove this, we need a result about the solution to \eqref{eq:linearSchrodinger} for initial data in $\mathcal{U}(\varepsilon,m_1,m_2)$. 

\begin{lem} 
Consider the set 
\[
\mathcal{U}_2 (\varepsilon,m_2) = \{(\phi_k)_{k\in\Z}\mid |\phi_k -\phi_k^{\ast}|\leq \varepsilon\quad \mbox{for all}\ |k|\leq m_2,\ \mbox{except $k=0,1$} \}.
\]
If $u$ is the solution to \eqref{eq:linearSchrodinger} with initial data given by $(r_k,\phi_k)_{k\in\Z}\in\mathcal{H}$ and $(\phi_k)_{k\in\Z}\in \mathcal{U}_2 (\varepsilon,m_2)$, then we have that
\begin{equation}\label{eq:truncation+error}
    \sup_{x\in\T}|u(t,x)|^2 \geq (1-2\varepsilon^2)\, \norm{ (c_k r_k)_{|k|\leq m_1}}_{\ell^1_k}^2 + E(\varepsilon,m_1, m_2),
\end{equation}
where the error term $E(\varepsilon,m_1, m_2):= E^+(\varepsilon,m_1, m_2)+E^-(\varepsilon,m_1, m_2)$, with
\begin{equation}\label{eq:truncation_error_positive}
\begin{split}
    E^+(\varepsilon,m_1,m_2)= \,(1&-2\varepsilon^2) \, \left(\norm{ (c_k r_k)_{|k|\leq m_2}}_{\ell^1_k}^2 - \norm{ (c_k r_k)_{|k|\leq m_1}}_{\ell^1_k}^2\right)
    \\ & \, +2\varepsilon^2\, \norm{ (c_k r_k)_{|k|\leq m_2}}_{\ell^2_k}^2 + \norm{ (c_k r_k)_{|k|> m_2}}_{\ell^2_k}^2
    \end{split}
\end{equation}
and
\begin{equation}\label{eq:truncation_error_negative}
E^-(\varepsilon,m_1,m_2) = -\sum_{(j,k)\in \mathcal{R}}c_j c_k r_j r_k,
\end{equation}
where $\mathcal{R}$ is defined as $\{(j,k)\in\Z^2: j\neq k,\ \mbox{and}\ |j|>m_2(\varepsilon) \ \mbox{or}\ |k|> m_2(\varepsilon)\}$.
\end{lem}

\begin{rk}\label{rk:ErrorIsPositive}
Notice that, when $m_1 \le m_2$, one has in fact that $E^+(\varepsilon,m_1,m_2) \ge 0$.
\end{rk}

\begin{proof}
As in \eqref{eq:overest}, we write 
\begin{equation}\label{eq:truncation_bound1}
\sup_{x\in\T}|u(t,x)|^2  = \sum_{k \in \Z} (c_k r_k)^2 + \sup_{x \in \T} \sum_{j\neq k} c_j c_k r_j r_k \cos(\psi_j -\psi_k)
\end{equation}
where $\psi_k=\psi_k (t,x) = \phi_k + k x - k t^2$. Let $\phi_0,\phi_1\in[0,2\pi)$ be free and choose $x^*$ so that 
\[
\phi_1 + x^* -t^2 = \phi_0.
\]
Then we define $\phi^{\ast}_k=\phi^{\ast}_k(\phi_0, \phi_1)$ for $k\neq 0,1$ according to \eqref{eq:MLE_phases}. For $|j|,|k|\leq m_2$, $j,k\neq 0,1$, we use the trivial bound $1- \cos y \leq y^2/2$ together with the fact that $(\phi_k)_{k\in\Z}\in \mathcal{U}_2 (\varepsilon,m_2)$ to estimate 
\begin{equation}\label{eq:truncation_bound2}
\begin{split}
 1- \sup_{x\in\T} \cos (\psi_j -\psi_k) & \leq \frac{1}{2}\, |\psi_j -\psi_k|^2 \Big |_{x=x^*} = \frac{1}{2}\, |\phi_j + j x^* - j t^2 - (\phi_k + k x^* - k t^2)|^2\\
 & = \frac{1}{2}\, | \phi_j +\phi_0^{\ast} - \phi_j^{\ast} - (\phi_k + \phi_0^{\ast}-\phi_k^{\ast})|^2 \\
 & = \frac{1}{2}\, |(\phi_j-\phi_j^{\ast})-(\phi_k-\phi_k^{\ast})|^2 \leq 2 \varepsilon^2.
 \end{split}
\end{equation}
The cases where $\{j,k\}\cap\{0,1\}\neq \emptyset$ satisfy the same estimate by construction. By \eqref{eq:truncation_bound1} and \eqref{eq:truncation_bound2},
\begin{align*}
\sup_{x\in\T}|u(t,x)|^2 \geq \sum_{k \in \Z} (c_k r_k)^2 + (1-2 \varepsilon^2) \sum_{j\neq k, |j|,|k|\leq m_2} c_j c_k r_j r_k - \sum_{(j,k)\in\mathcal{R}} c_j c_k r_j r_k
\end{align*}
Finally, one can rearrange the terms on the right-hand side to obtain \eqref{eq:truncation+error}.
\end{proof}

\noindent In view of \eqref{eq:condU1} and \eqref{eq:truncation+error}, if $\vartheta\in\mathcal{U}(\varepsilon,m_1,m_2)$, it follows that
\begin{equation}
    \sup_{x\in\T}|u(t,x)|^2 \geq (1-2\varepsilon^2)\, \left(\frac{\sum_{|k|\leq m_1(\varepsilon)}c_k^2}{\sum_{k\in\Z} c_k^2}\, z_0 \varepsilon^{-1/2}\right)^2 + E(\varepsilon,m_1,m_2)
\end{equation}
It is easy to check that such $\vartheta\in\mathcal{U}(\varepsilon,m_1,m_2)$ will live in $\mathcal{D}(t,z(\varepsilon)-a(\varepsilon))$ as long as the error is not too small, namely
\begin{equation}\label{eq:E_condition}
\begin{split}
E(\varepsilon,m_1,m_2) \geq [z_0\varepsilon^{-1/2}-a(\varepsilon)]^2- (1-2\varepsilon^2)\, \left(\frac{\sum_{|k|\leq m_1(\varepsilon)}c_k^2}{\sum_{k\in\Z} c_k^2}\, z_0 \varepsilon^{-1/2}\right)^2 =: f(\varepsilon).
\end{split}
\end{equation}

\noindent Our next goal is to show that the error satisfies this inequality with very high probability.

\begin{lem}\label{thm:error_dexp}
Suppose that $m_1$ and $m_2$ are chosen as in \Cref{thm:LDP_U}. Then for any $c>0$ the error term in \eqref{eq:truncation_error_positive} and \eqref{eq:truncation_error_negative} satisfies
\begin{equation}
\log \P \left( E(\varepsilon,m_1,m_2)) < -c\,\varepsilon \right) \lesssim - c\,\varepsilon\, \exp(b\,\varepsilon^{-\alpha}), \qquad \mbox{as}\ \varepsilon\rightarrow 0^{+}.
\end{equation}
\end{lem}
\begin{rk}\label{rk:error_dexp} Note that for $f$ as in \eqref{eq:E_condition}, we can choose $a(\varepsilon)$ in such a way that the following holds.
\[
\P \left( E(\varepsilon,m_1,m_2)) < f(\varepsilon) \right) \leq \P \left( E(\varepsilon,m_1,m_2)) \leq -c\varepsilon \right)
\]
This is due to the fact that:
\begin{equation}\label{eq:asymp_f}
    f(\varepsilon)\lesssim - \varepsilon^{2\beta-} \quad\mbox{for}\ \beta\in (0,1/2).
\end{equation}
In order to prove \eqref{eq:asymp_f}, we first write:
\begin{equation}\label{eq:error_f}
    \begin{split}
        f(\varepsilon) = &\ -2z_0\, \varepsilon^{-1/2}\, a(\varepsilon) + z_0^2 \varepsilon^{-1}\left( 1- \frac{\norm{(c_k)_{|k|\leq  m_1(\varepsilon)}}_{\ell^2_k}^4}{\norm{(c_k)_{k\in\Z}}_{\ell^2_k}^4}\right) \\
        & +a(\varepsilon)^2 + 2\varepsilon^2 \, \left(\frac{\sum_{|k|\leq m_1(\varepsilon)}c_k^2}{\sum_{k\in\Z} c_k^2}\, z_0 \varepsilon^{-1/2}\right)^2.
    \end{split}
\end{equation}
We want to choose $a(\varepsilon)$ in such a way that $f$ is eventually negative. To do so, we must identify the top order of \eqref{eq:error_f} as $\varepsilon\rightarrow 0^{+}$. First of all, note that the second summand in \eqref{eq:error_f} has size $\O (\varepsilon^{2\beta})$, since
\[
1- \frac{\norm{(c_k)_{|k|\leq  m_1(\varepsilon)}}_{\ell^2_k}^4}{\norm{(c_k)_{k\in\Z}}_{\ell^2_k}^4} = \norm{(c_k)_{|k|>m_1(\varepsilon)}}_{\ell^2_k}^2 \, \frac{\norm{(c_k)_{|k|\leq m_1(\varepsilon)}}_{\ell^2_k}^2+\norm{(c_k)_{k\in\Z}}_{\ell^2_k}^2}{\norm{(c_k)_{k\in\Z}}_{\ell^2_k}^4} \sim c_{m_1(\varepsilon)}^2 \cdot 1 \sim \varepsilon^{1+2\beta},
\]
which follows from \eqref{eq:m1} and the choice of $(c_k)_{k \in \Z}$. Similarly, the last summand in \eqref{eq:error_f} has size $\O (\varepsilon)$. As a result, we can choose $a(\varepsilon)=\varepsilon^{1/2+2\beta-}$, so that the first summand in \eqref{eq:error_f} dominates. Then \eqref{eq:asymp_f} follows right away.
\end{rk}
\begin{proof}
For $\lambda>0$ to be fixed, we can bound the desired probability using Markov inequality:
\begin{align}\label{eq:chernoff_error}
\P \left( E(\varepsilon,m_1, m_2)) \leq -c\,\varepsilon \right) & =\P \left( e^{-\lambda E(\varepsilon, m_1, m_2))} \geq e^{\lambda c \varepsilon} \right) \leq e^{-\lambda c\varepsilon}\, \E[e^{-\lambda E(\varepsilon,m_1, m_2)}]\nonumber\\
& \leq e^{-\lambda c \varepsilon}\, \E [e^{-2 \lambda E^+(\varepsilon,m_1, m_2)}]^{1/2}\, \E [e^{-2 \lambda E^-(\varepsilon,m_1, m_2)}]^{1/2}.
\end{align}
The last inequality follows from the Cauchy-Schwartz inequality, \eqref{eq:truncation_error_positive}, and \eqref{eq:truncation_error_negative}. 

Let us estimate the second factor in \eqref{eq:chernoff_error} by 1, since $E^{+}(\varepsilon,m_1,m_2)$ is always non-negative by \eqref{eq:truncation_error_positive} and \Cref{rk:ErrorIsPositive}. Next we estimate the third factor in \eqref{eq:chernoff_error}. We aim to show that it is uniformly bounded in $\varepsilon$ as long as we choose $\lambda$ correctly. To show this, we first bound the exponent:
\[
\begin{split}
2\lambda \sum_{(j,k)\in \mathcal{R}}c_j c_k r_j r_k & \leq \lambda \sum_{(j,k)\in\mathcal{R}} c_j c_k (r_j^2 + r_k^2) \\
& \leq 2\lambda \sum_{|k|\leq m_2(\varepsilon)} c_k r_k^2 \sum_{|j|>m_2(\varepsilon)}c_j  + 2\lambda \sum_{|k|>m_2(\varepsilon)} c_k r_k^2 \sum_{j\neq k}c_j \\
& = 4\lambda\, \frac{e^{-b(m_2(\varepsilon)+1)}}{1-e^{-b}} \,\sum_{|k|\leq m_2(\varepsilon)}c_k r_k^2 + 2\lambda\, \frac{1+e^{-b}}{1-e^{-b}} \, \sum_{|k|> m_2(\varepsilon)} c_k r_k^2.
\end{split}
\]
Using this bound and independence, we have that
\begin{align}\label{eq:ring_terms}
\E \left[e^{2\lambda \sum_{(j,k)\in \mathcal{R}}c_j c_k R_j R_k}\right] & \leq \E \left[ e^{4\lambda\, \frac{e^{-b(m_2(\varepsilon)+1)}}{1-e^{-b}} \,\sum_{|k|\leq m_2(\varepsilon)}c_k R_k^2 }\right] \, \E \left[ e^{2\lambda\, \frac{1+e^{-b}}{1-e^{-b}} \, \sum_{|k|> m_2(\varepsilon)} c_k R_k^2}\right]\nonumber\\
& = \prod_{|k|\leq m_2(\varepsilon)} \frac{1}{1-4\lambda\, \frac{e^{-b(m_2(\varepsilon)+1)}}{1-e^{-b}} c_k} \, \prod_{|k|>m_2(\varepsilon)} \frac{1}{1-2\lambda\, \frac{1+e^{-b}}{1-e^{-b}} c_k}.
\end{align}

Let us bound the first factor in \eqref{eq:ring_terms}. Fix $\lambda = C \exp (bm_2(\varepsilon))=C\exp(b\varepsilon^{-\alpha})$ for some $C>0$ to be chosen soon. Since $|k|\leq m_2(\varepsilon)$, we have that $\lambda e^{-b(m_2(\varepsilon)+1)}\leq C e^{-b} < C$. Choose $C$ small enough so that 
\begin{equation}
\frac{4C}{1-e^{-b}} <\frac{1}{2}. 
\end{equation}
This choice guarantees that
\[
\log \left( \prod_{|k|\leq m_2(\varepsilon)} \frac{1}{1-4\lambda\, \frac{e^{-b(m_2(\varepsilon)+1)}}{1-e^{-b}} c_k}\right) \leq \sum_{|k|\leq m_2(\varepsilon)} \log \left( \frac{1}{1-c_k/2}\right)\leq \sum_{|k|\leq m_2(\varepsilon)} c_k \lesssim 1.
\]
Consequently, the first factor in \eqref{eq:ring_terms} is uniformly bounded in $\varepsilon$. We perform a similar analysis with the second factor in \eqref{eq:ring_terms}. By making $C$ smaller (if necessary), we can arrange:
\[ 
2\lambda\, \frac{1+e^{-b}}{1-e^{-b}} c_k = 2 C e^{b\,m_2(\varepsilon)}\, \frac{1+e^{-b}}{1-e^{-b}}\, e^{-b \,m_2(\varepsilon) -b \,(|k|-m_2(\varepsilon))} \leq \frac{1}{2} e^{-b (|k|-m_2(\varepsilon))}.
\]
As a consequence,
\[
\log \left(\prod_{|k|>m_2(\varepsilon)} \frac{1}{1-2\lambda\, \frac{1+e^{-b}}{1-e^{-b}} c_k}\right) \leq \sum_{|k|>m_2(\varepsilon)}\log \left( \frac{1}{1-  \frac{1}{2} e^{-b (|k|-m_2(\varepsilon))}}\right) \leq \sum_{|k|>m_2(\varepsilon)} e^{-b (|k|-m_2(\varepsilon))} \lesssim 1
\]
uniformly in $\varepsilon$. Therefore we have that all the terms in \eqref{eq:ring_terms} are uniformly bounded in $\varepsilon$.

\noindent This result, together with \eqref{eq:chernoff_error}, yields
\[
\P \left( E(\varepsilon,m)) < -c\,\varepsilon \right) \leq C_b\, e^{-c \lambda\varepsilon} \lesssim_b \exp (-c\,\varepsilon\exp (b\, \varepsilon^{-\alpha}))
\]
and the lemma follows.
\end{proof}

As we mentioned in \Cref{rk:error_dexp}, \Cref{thm:error_dexp} shows that our one-sided neighborhood of minimizers, $\mathcal{U}(\varepsilon,m_1,m_2)$, is almost entirely contained in the set 
\[
\mathcal{D} \left(t,z(\varepsilon) - \varepsilon^{\frac{1}{2}+2\beta-}\right)= \left\lbrace \vartheta\in\mathcal{H} : \sup_{x\in\T} |u(t,x\, |\,\vartheta)|\geq z_0 \varepsilon^{-1/2}  - \varepsilon^{\frac{1}{2}+2\beta-}\right\rbrace
\]
for $\beta\in (0,1/2)$ and $m_1$ and $m_2$ as in \eqref{eq:m1}. More precisely, we have shown that the difference between the two events is doubly exponentially unlikely, i.e. there exists some $c>0$ such that
\[
\P \left( \mathcal{U}(\varepsilon,m_1,m_2) -\mathcal{D} \left(t,z(\varepsilon) - \varepsilon^{\frac{1}{2}+2\beta-}\right)\right) \ll \exp(- \exp (c \varepsilon^{-\alpha})), \qquad \mbox{as}\ \varepsilon\rightarrow 0^{+}.
\]

We summarize the results of this section in the following theorem. For the sake of clarity, we only present a special case of our results where we fix $\alpha=1/2$ and $\beta=2/5$.

\begin{thm}\label{thm:concentration}
Consider the set $\mathcal{U}(\varepsilon)$ given by the $\vartheta = (r_k, \phi_k)_{k\in\Z}\in\mathcal{H}$ satisfying:
\begin{equation}
    r^*_k \leq r_k \leq r^*_k + \varepsilon^{2/5}, \qquad \mbox{for all}\ |k|\leq m_1(\varepsilon) = -\frac{1}{b} \log \left( \frac{2\sum_{j\in\Z} c_j^2}{z_0}\, \varepsilon^{9/10}\right),
\end{equation}
and
\begin{equation}
 |\phi_k -\phi_k^{\ast}|\leq \varepsilon, \qquad \mbox{for all}\ |k|\leq m_2(\varepsilon)=\varepsilon^{-1/2}\ \mbox{except $k=0,1$},
\end{equation}
where $r^*_k$ and $\phi_k^{\ast}=\phi_k^{\ast}(\phi_0,\phi_1)$ were defined in \eqref{eq:optimal_r_k} and \eqref{eq:MLE_phases}, respectively. Then $\mathcal{U}(\varepsilon)$ concentrates asymptotically as much probability as the set of rogue waves. Moreover, $\mathcal{U}(\varepsilon)$ is almost entirely contained in the set  $\mathcal{D} (t,z_0\varepsilon^{-1/2}-\varepsilon)$. More precisely,
\begin{equation}
\log \P \left( \mathcal{U}(\varepsilon) -\mathcal{D} (t,z_0\varepsilon^{-1/2} - \varepsilon)\right) \lesssim - \exp (c \varepsilon^{-1/2})) \qquad \mbox{as}\ \varepsilon\rightarrow 0^{+}.
\end{equation}
\end{thm}
\begin{rk}\label{rk:fromLineartoNLS}
The construction of the set $\mathcal{U}(\varepsilon)$ in \Cref{thm:concentration} is based on the linear Schr\"odinger equation. However, due to the method we used to prove \Cref{thm:mainLDP}, it is possible to adapt\footnote{Note that $(r_k^{\ast}, \phi^*_k)_{k \in \Z}$ may not solve the minimization problem associated with \eqref{eq:NLS_mainLDP} anymore, which is analogous to \eqref{eq:mainMinProblem}. Despite that, the proof of \Cref{thm:mainLDP} shows that the linear part of the equation dominates for such times, and therefore the $(r_k^{\ast}, \phi^*_k)_{k \in \Z}$ for the linear equation are a good approximation to the actual minimizers.} \Cref{thm:concentration} to handle subcritical times of the nonlinear equation \eqref{eq:NLS_mainLDP}. 

Regarding critical times, it is possible to construct an analogous set $\mathcal{U}_{\mathrm{app}}(\varepsilon)$ based on the resonant approximation in \eqref{eq:resonant_app}. To do so, one must change the definition of $\phi_k^{\ast}$ in \eqref{eq:MLE_phases} to include the additional terms appearing in the phases in \eqref{eq:resonant_app}, namely
\[
\phi_k^{\ast} :=\phi^*_0 - kx^* + k^2 t - \varepsilon^2 t\, (c_k r_k^{\ast})^2 \, \left(\text{mod } 2\pi\right).
\]
\end{rk}
\begin{rk}
Beyond \Cref{thm:concentration}, an interesting open question would be to investigate the difference $\mathcal{D} (t,z_0\varepsilon^{-1/2} -\varepsilon)-\mathcal{U}(\varepsilon)$. In particular, is it possible to prove that for some choice of $m_1$ and $m_2$ we have that 
\[
\P \left( \mathcal{D} (t,z(\varepsilon))- \mathcal{U}(\varepsilon,m_1,m_2)\right) \ll \P \left( \mathcal{D} (t,z(\varepsilon))\right) \quad \mbox{as}\ \varepsilon\rightarrow 0^{+}\, ?
\]
This would indicate that any rogue wave that doesn't belong to the set $\mathcal{U}$ is much more unlikely to happen than those we can characterize in $\mathcal{U}$. In order to prove such a result, one might need a better understanding of the asymptotic behavior of $\P \left( \mathcal{D} (t,z(\varepsilon))\right)$ beyond the top order (which is all we need in the proof of \Cref{thm:mainLDP}).
\end{rk}

\begin{rk}\label{rk:growth}
Another interesting open question is that of quantifying the growth of rogue waves. Suppose that we have a rogue wave of height $z>0$ at time $t>0$. What is the smallest (in the sense of $L^{\infty}_x$) initial datum $\vartheta\in\mathcal{H}$ that could have given rise to a wave of such height? And for such $\vartheta$, what is the growth coefficient? Mathematically, that would be
\[
\frac{\sup_{x\in\T} |u(t,x\, |\, \vartheta)|}{\sup_{x\in\T}  |u(0,x\, |\,\vartheta)|} = \frac{z}{\sup_{x\in\T}  |u(0,x\, |\, \vartheta)|}.
\]
We investigate this question numerically in the next subsection.
\end{rk}

\subsection{Numerical Examples}

\begin{figure}
    \centering
    \includegraphics[width=4.5in]{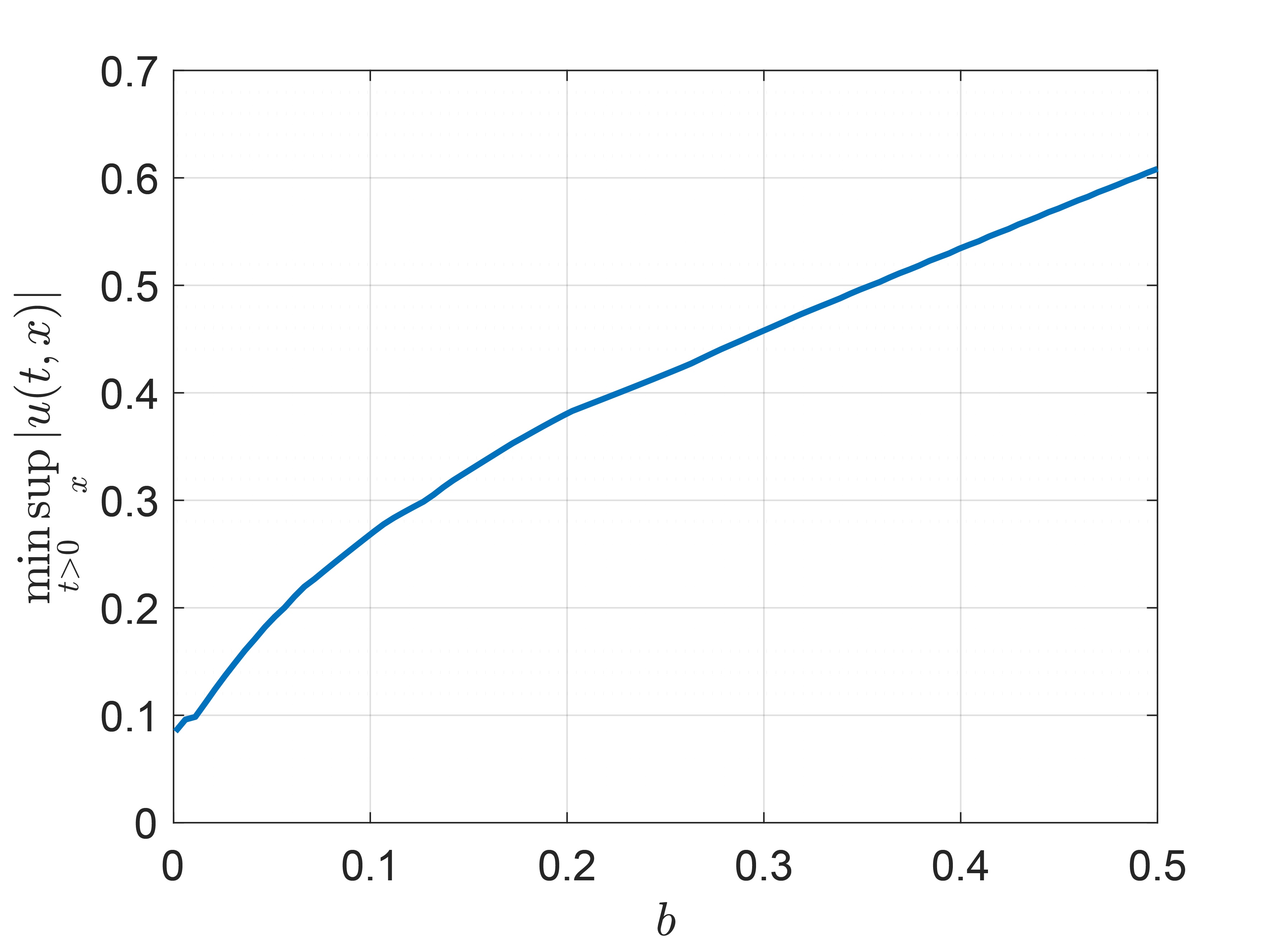}
    \caption{Dependence of $\mathrm{m}$ as defined in \eqref{eq:minimax_simulation} on $b$ for $u(t,x)$ given by \eqref{eq:sol_simulation} where $n=100$, $z=1$, ${0\leq t\leq 2\pi}$, ${-\pi\leq x\leq \pi}$, and using $N=2500$ grid points each for $x$ and $t$ (i.e., $\Delta t=\Delta x\approx 2.5\times 10^{-3}$). The simulations consider $100$ $b$-values in the range $10^{-3}\leq b\leq 5\times 10^{-1}$ with $\Delta b\approx 5\times 10^{-3}$.}
    \label{fig:mvsb}
\end{figure}

\begin{figure}
    \centering
    \includegraphics[width=5.5in]{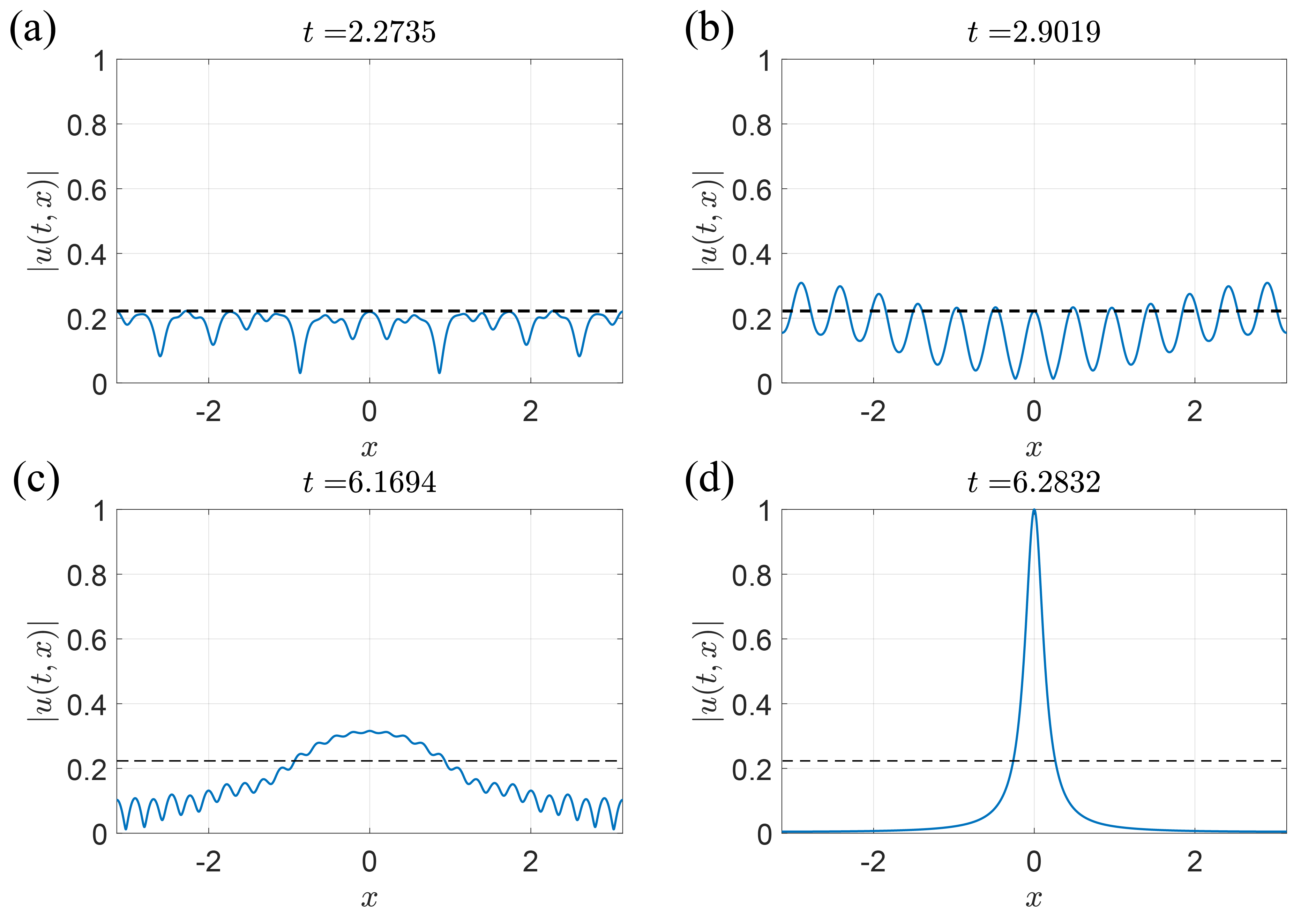}
    \caption{Snapshots of numerical simulations of \eqref{eq:sol_simulation} at times: (a) $t\approx2.2735$, (b) $t\approx2.9019$, (c) $t\approx6.1694$, and (d) $t\approx6.2832$. The solution was simulated for ${0\leq t\leq 2\pi}$, ${-\pi\leq x\leq \pi}$ using $z=1$, $b=0.07$, $n=100$, and $N=10^4$ grid points each for $x$ and $t$ (i.e., $\Delta t=\Delta x\approx 6.28\times 10^{-4}$). The dashed line in each image represents the value $\displaystyle{\mathrm{m}\approx 0.2225}$ as defined in \eqref{eq:minimax_simulation}, which is obtained in (a).}
    \label{fig:snapshots}
\end{figure}

In this section we investigate the question of maximum possible growth of rogue waves as described in \Cref{rk:growth}. The main question is the following: which is the smallest initial datum $\vartheta\in\mathcal{H}$ that gives rise to a rogue wave of height $z>0$ at time $t_0>0$. This smallness is measured in terms of the initial height: $\sup_{x\in\T} |u(0,x\,|\, \vartheta)|$.

In the case of the linear Schr\"odinger equation, we focus on the minimizers obtained in \Cref{prop:existence_minimizer}. Using the symmetries of the Schr\"odinger equation, we may assume that the peak of the rogue wave happens at $x^{\ast}=0$, with phase $\phi_0^{\ast}=0$, at time $t=0$. Moreover, we may normalize the height so that $z=1$. This initial datum is given by
\begin{equation}\label{eq:datum_simulation}
u_0^{\ast}(x)= \sum_{k\in\Z} \frac{c_k^2}{\sum_{j\in\Z} c_j^2} \, e^{ikx}.
\end{equation}

Then the question of maximum possible growth is equivalent to finding the following quantity:
\begin{equation}\label{eq:minimax_simulation}
\mathrm{m}=\min_{t\geq 0} \sup_{x\in\T} |u(t,x)|.
\end{equation}

We conduct numerical simulations to find \eqref{eq:minimax_simulation} in the case where $u(t,x)$ is the solution to the linear Schr\"odinger equation with initial datum \eqref{eq:datum_simulation}. To do so, we must truncate the series in \eqref{eq:datum_simulation}, so we consider
\begin{equation}\label{eq:datum_simulation2}
u_0^{\ast}(x)= \sum_{|k|\leq n} \frac{c_k^2}{\sum_{|j|\leq n} c_j^2} \, e^{ikx}
\end{equation}
and the associated solution
\begin{equation}\label{eq:sol_simulation}
u(t,x)= \sum_{|k|\leq n} \frac{c_k^2}{\sum_{|j|\leq n} c_j^2} \, e^{ikx-it|k|^2}.
\end{equation}

We were first interested in exploring the dependence of growth on the constant $b$ appearing in the exponent of $c_k$ in \eqref{eq:intro_data}. We therefore computed $u(t,x)$ from \eqref{eq:sol_simulation}
with $n=100$, $-\pi\leq x\leq \pi$, and $0\leq t\leq 2\pi$ using 2500 data points each for $x$ and $t$ (i.e., $\Delta t=\Delta x\approx 2.5\times 10^{-3}$). We then computed $\mathrm{m}$ in \eqref{eq:minimax_simulation} for $100$ values of $b$ in the range $10^{-3}\leq b\leq 5\times 10^{-1}$ and note that $\mathrm{m}=\mathrm{m}(b)$ generally decreases as $b$ is decreased (Figure \ref{fig:mvsb}). Note that Figure \ref{fig:mvsb} seems to break down as $b\to 0^+$. We suspect that this is due to the use of $n=100$. Equations \eqref{eq:condU1} and \eqref{eq:m1} seem to indicate that one needs to control a certain number of Fourier modes for the truncation error to stay small. This number grows proportionally to $1/b$, so we might expect that for $b\lesssim 0.01$ (up to an unknown constant) this number would need to be larger than 100.

As a demonstration of this growth, we chose $b=0.07$ and used a refined discretization of $\Delta t=\Delta x\approx 6.28\times 10^{-4}$ to capture dynamics on fast time scales. This computation yields $\mathrm{m}\approx 0.2225$, representing a wave whose height can grow by a factor of approximately $4.5$. Snapshots of the simulation are shown in Figure \ref{fig:snapshots}. The dashed black line indicates the value of $\mathrm{m}$, which is obtained in Figure \ref{fig:snapshots}(a).

\bibliographystyle{amsplain}
\bibliography{references}
\end{document}